\documentclass[preprint,12pt]{elsarticle}

\usepackage{amsthm}
\usepackage{amsmath,amssymb,txfonts}
\usepackage{color,bm,comment}
\usepackage[english]{babel}

\newtheorem{theorem}{Theorem}
\newtheorem{lemma}{Lemma}
\newtheorem{prop}{Proposition}
\newtheorem{cor}{Corollary}
\newtheorem{remark}{Remark}

\theoremstyle{definition}

\def\re{\mathbb{R}}
\def\N{\mathbb{N}}

\def\({\left(}
\def\){\right)}
\def\[{\left[}
\def\]{\right]}
\def\pd{\partial}
\def\lap{\Delta}

\def\ep{\varepsilon}
\def\w{\omega}
\def\la{\lambda}

\def\al{\alpha}
\def\ol{\overline}

\def\weakto{\rightharpoonup}

\newtheorem{ThmA}{Theorem A}

\begin{document}

\begin{frontmatter}



\title{On eigenvalue problems involving the critical Hardy potential and Sobolev type inequalities with logarithmic weights in two dimensions}


\author[MS]{Megumi Sano}
\ead{smegumi@hiroshima-u.ac.jp}

\author[FT]{Futoshi Takahashi\corref{Takahashi}\fnref{label1}}
\ead{futoshi@omu.ac.jp}
\fntext[label1]{Corresponding author.}

\address[MS]{Laboratory of Mathematics, School of Engineering,
Hiroshima University, Higashi-Hiroshima, 739-8527, Japan}
\address[FT]{Department of Mathematics, Graduate School of Science, Osaka Metropolitan University, Sumiyoshi-ku, Osaka, 558-8585, Japan}


\date{\today}

\begin{abstract}
We consider the two-dimensional eigenvalue problem for the Laplacian with the Neumann boundary condition involving the critical Hardy potential. 
We prove the existence of the second eigenfunction and study its asymptotic behavior around the origin.
A key tool is the Sobolev type inequality with a logarithmic weight, which is shown in this paper as an application of the weighted nonlinear potential theory.
\end{abstract}

\end{frontmatter}



%
%
\section{Introduction}\label{S Intro}

Let $\Omega \subset \re^2$ be a smooth bounded domain with $0 \in \overline{\Omega}$. 
For simplicity, we assume $\sup_{x \in \Omega} |x| =1$ without loss of generality. 
Let $a \ge 1$.
In this paper, we consider the following linear eigenvalue problem 
\begin{align*}
(N)\,
	\begin{cases}
	-\Delta u = \la \frac{u}{|x|^2 \( \log \frac{a}{|x|} \)^2}\quad &\text{in}\,\, \Omega, \\
	\frac{\pd u}{\pd \nu} = 0 \quad &\text{on} \,\, \pd \Omega,
	\end{cases}
\end{align*}
here $\nu$ denotes the unit outer normal vector to $\pd\Omega$.
The problem stems from the critical Hardy inequality on a bounded domain $\Omega \subset \re^2$
for functions in the Sobolev space $H^1_0(\Omega)$: for any $u \in H^1_0(\Omega)$, it holds that
\begin{equation}
\label{H^1_0 CH}
	\frac{1}{4} \int_{\Omega} \frac{u^2}{|x|^2 \( \log \frac{a}{|x|} \)^2} \,dx \le \int_{\Omega} |\nabla u |^2 \,dx. 
\end{equation}
Moreover, the constant $\frac{1}{4}$ on the left-hand side is best possible and is not attained. 
We recall that the Sobolev space $H^1(\Omega)$ is a set of functions $u \in L^2(\Omega)$ such that its distributional gradient $\nabla u$ is also in $L^2(\Omega)$.
$H^1(\Omega)$ is a Hilbert space with an inner product $(u, v)_{H^1(\Omega)} = \int_{\Omega} ( \nabla u \cdot \nabla v + u v ) \, dx$, and
$H^1_0(\Omega)$ is a closure of $C_0^{\infty}(\Omega)$ with respect to the norm $\| u \|_{H^1(\Omega)} = (u, u)_{H^1(\Omega)}^{1/2}$. 
For the inequality \eqref{H^1_0 CH}, we refer the readers to \cite{II}, \cite{S(JDE)}, \cite{ST} \cite{ST(HT)} and the references there in.

In a higher dimensional case, we know the subcritical Hardy inequality for functions in $H^1_0(\Omega)$:
\begin{equation*}
	H_N \int_{\Omega} \frac{u^2}{|x|^2} \,dx \le \int_{\Omega} |\nabla u |^2 \,dx 
\end{equation*}
holds for any $u \in H^1_0(\Omega)$, here $\Omega$ is a bounded domain in $\re^N$, $N \ge 3$, with $0 \in \Omega$.
The constant $H_N = \(\frac{N-2}{2}\)^2$ is optimal and is never attained by a non zero function in $H^1_0(\Omega)$.
In \cite{CPR}, Chabrowski, Peral and Ruf consider the linear eigenvalue problem
\begin{align*}
	\begin{cases}
	-\Delta u = \la \frac{u}{|x|^2}\quad &\text{in}\,\, \Omega \subset \re^N, N \ge 3, \\
	\frac{\pd u}{\pd \nu} = 0 \quad &\text{on} \,\, \pd \Omega.
	\end{cases}
\end{align*}
Clearly, the first eigenvalue is $\la = 0$ and constant functions are the first eigenfunctions.
To seek the nontrivial solution in $H^1(\Omega)$, the authors in \cite{CPR} introduce the minimization problem
\begin{align*}
	\la_H = \inf \left\{  \,\frac{\int_{\Omega} |\nabla u|^2\,dx }{\int_{\Omega} \frac{|u|^2}{|x|^2} \,dx} \,\,\middle| \,\, u \in H^1 (\Omega) \setminus \{ 0\}, 
	\, \int_{\Omega} \frac{u}{|x|^2} \,dx = 0\, \right\}, 
\end{align*}
and prove that if $\la_H < H_N$, then $\la_H$ is attained and the minimizer corresponds the second eigenfunction of the above eigenvalue problem. 
Also the authors obtain several examples of domains such that the condition $\la_H < H_N$ holds true. 
Especially, they establish the existence of the second eigenfunction on balls in $\re^N$, $N \ge 7$.
Moreover, they study the asymptotic behavior of the second eigenfunctions around the origin in the case $0 \in \Omega$.
To obtain the asymptotic estimate of the second eigenfunction near the origin, they use De Giorgi-Nash-Moser type procedure and the Caffarelli-Kohn-Nirenberg inequality \cite{CKN}.

The aim of this paper is to extend the results in \cite{CPR} to the two-dimensional problem (N).
Let $a \ge 1$ and $0 \in \overline{\Omega} \subset \re^2$. 
We consider the minimization problem
\begin{align}
\label{la_a}
	\la_a = \inf \left\{  \,\frac{\int_{\Omega} |\nabla u|^2\,dx }{\int_{\Omega} \frac{|u|^2}{|x|^2 \( \log \frac{a}{|x|} \)^2} \,dx} \,\,\middle| \,\, u \in H^1 (\Omega) \setminus \{ 0\}, 
	\, \int_{\Omega} \frac{u}{|x|^2 \( \log \frac{a}{|x|} \)^2} \,dx = 0\, \right\}. 
\end{align}
We show $\la_a > 0$ for $a > 1$, see (\ref{average CH}). 
We seek for sufficient conditions to assure the existence of minimizers, which yiedls the second eigenfunction of the problem (N).
Our sufficient condition claims that if $a > 1$ and $\la_a < \frac{1}{4}$, then $\la_a$ is attained by a non trivial function in $H^1(\Omega)$. 
We also study the asymptotic behavior near the origin of the second eigenfunctions.
We remark that, unlike \cite{CPR}, we can treat the case $0 \in \pd\Omega$ too.
Furthermore, since our Hardy potential involves the logarithmic weights, 
it is difficult to control the weights by Caffarelli-Kohn-Nirenberg type inequality, which was useful for treating power type weights.
Therefore, we need to establish the Sobolev type inequality with logarithmic weights.
Combining this inequality with the De Giorgi-Nash-Moser procedure, we obtain the expected asymptotic behavior of the second eigenfunctions.
To obtain the Sobolev type inequality with logarithmic weights, we exploit weighted nonlinear potential theory by \cite{Adams}.
We believe that this part is also interesting in itself.

In the following, $L^s(\Omega)$ will denote the standard Lebesgue spaces.  
Also for a given nonnegative weight function $\w$ and $1 \le s < \infty$, the weighted Lebesgue space $L^s(\Omega, \w(x)dx)$ is the set of functions $u$ such that $\int_{\Omega} |u|^s w(x) dx < \infty$.
$B_r$ will denote a ball in $\re^2$ of radius $r$ with center the origin.
``$\to$" and ``$\weakto$" will denote the strong and weak convergence in Banach spaces, respectively.
(Possibly different) general positive constants are denoted by $C$.

%
%
\section{The critical Hardy type inequalities for $H^1(\Omega)$.}\label{S CH}

Let $\Omega \subset \re^2$ be a smooth bounded domain with $0 \in \overline{\Omega}$ and $\sup_{x \in \Omega} |x| =1$.
In \S \ref{S Sob log}, we prove the following Hardy-Sobolev type inequality with logarithmic weights.

%
%
\begin{theorem}
\label{Thm ineq log}
Let $a>1$, $p \ge 2$, $B<1$, $A \ge 1+ \frac{p}{2} (1-B)$.
Then there exists a positive constant $C_{p, A, B}$ such that the inequality
\begin{align}\label{ineq log}
	C_{p, A, B} \( \int_{\Omega} \frac{|u|^p}{|x|^2 (\log \frac{a}{|x|})^A} \,dx \)^{\frac{2}{p}} \le \int_\Omega \( \log \frac{a }{|x|} \)^{B} |\nabla u|^2 \,dx
\end{align}
holds for any $u \in C_c^\infty (\Omega)$.
\end{theorem}
Note that if we set $p = 2$, $A = 2$, $B = 0$, then \eqref{ineq log} is nothing but the critical Hardy inequality \eqref{H^1_0 CH}.
Also we remark that if $A=1+ \frac{p}{2} (1-B)$ and $\Omega = B_1$, 
then the inequality (\ref{ineq log}) has the scale invariance under the following scaling 
$C_c^\infty (B_1) \ni u \mapsto u_{\la}$ for any $\la \le 1$, where 
\begin{align*}
	u_\la (x) = 
	\begin{cases}
	\la^{-\frac{1-B}{2}} u \( \( \frac{|x|}{a} \)^{\la -1} x \) 
	\,\,&\text{for}\,\, x \in B_{a^{\frac{\la-1}{\la}}}, \\
	0 &\text{for} \,\, x \in B_1 \setminus B_{a^{\frac{\la-1}{\la}}}.
\end{cases}
\end{align*}
When $A< 1+ \frac{p}{2} (1-B)$, \eqref{ineq log} does not have this scale invariance and by letting $\la \to 0$, 
we can easily show that the inequality does not hold when $A< 1+ \frac{p}{2} (1-B)$. 

The proof of Theorem \ref{Thm ineq log} is postponed to \S \ref{S Sob log}.

\vspace{1em}
To prove the existence of the second eigenfunctions for the problem (N), 
we need the critical Hardy inequality for functions in $H^1(\Omega)$, and also for functions in $H^1(\Omega)$ with average zero.
Also to treat the case $0 \in \pd\Omega$, we need the next lemma.

%
%
\begin{lemma}\label{Lemma ineq log}
Let $x= (x_1, x_2) \in \re^2$, $h \in C^1 (\re)$, $0 < r \le 1$ and $h(0) = 0$, and $h'(0) = 0$. 
Set $B^h_r = B_r \cap \{ x \in \re^2 \,| \, x_2 > h (x_1) \}$.
Let $a, p, A, B$ as in Theorem \ref{Thm ineq log}.
Then, for any $\ep > 0$, there exists $\delta >0$ such that if $|h'(x_1)| \le \delta$ for any $x_1 \in (-r, r)$, 
then the inequality
\begin{align*}
	2^{\frac{2}{p} -1} C_{p, A, B} \( \int_{B_r^h} \frac{|u|^p}{|x|^2 (\log \frac{a}{|x|})^A} \,dx \)^{\frac{2}{p}} 
	\le (1+ \ep) \int_{B_r^h} \( \log \frac{a }{|x|} \)^{B} |\nabla u|^2 \,dx
\end{align*}
holds for any $u \in H^1 (B_r)$ with ${\rm supp}\, u \subset B_r$, where $C_{p, A, B}$ is given in Theorem \ref{Thm ineq log}.
Especially, we have
\begin{align*}
	\int_{B^h_r} \frac{u^2}{|x|^2 \( \log \frac{a}{|x|} \)^2} \,dx \le (4+\ep) \int_{B^h_r} |\nabla u |^2 \,dx
\end{align*}
for any $u \in H^1 (B_r)$ with ${\rm supp}\, u \subset B_r$.
\end{lemma}

\begin{proof}
We follow the argument of the proof of \cite{HL} Lemma 2.1. 

\noindent
(I)\hspace{1em} Assume that $h(x_1) \equiv 0$. 
Since the value of $u(x)$ are irrelevant for $x_2 < 0$, we may suppose that $u(x)$ is even in $x_2$. 
By the inequality (\ref{ineq log}) in Theorem \ref{Thm ineq log}, 
we have
\begin{align*}
	C_{p, A, B} \( \int_{B^0_r} \frac{|u|^p}{|x|^2 \( \log \frac{a}{|x|} \)^A} \,dx \)^{\frac{2}{p}} 
	&=C_{p, A, B} \( \frac{1}{2} \int_{B_r} \frac{|u|^p}{|x|^2 \( \log \frac{a}{|x|} \)^A} \,dx \)^{\frac{2}{p}}  \\
	&\le 2^{-\frac{2}{p}} \int_{B_r} \( \log \frac{a}{|x|} \)^B |\nabla u |^2 \,dx\\
	&= 2^{1-\frac{2}{p}} \int_{B^0_r} \( \log \frac{a}{|x|} \)^B |\nabla u |^2 \,dx.
\end{align*}

\noindent
(II)\hspace{1em} 
We consider the case where $h(x_1) \not\equiv 0$. 
Then we set $y_1 = x_1, y_2 = x_2 - h(x_1)$ and $\tilde{u} (y_1, y_2) = u(x_1, x_2)$. 
From (I), we have
\begin{align}\label{tilde u}
	 2^{\frac{2}{p} -1} C_{p, A, B} \( \int_{B^0_r} \frac{|\tilde{u}|^p}{|y|^2 \( \log \frac{a}{|y|} \)^A} \,dy \)^{\frac{2}{p}} \le \int_{B^0_r}\( \log \frac{a}{|y|} \)^B |\nabla \tilde{u} |^2 \,dy.
\end{align}
Direct calculation implies that 
\begin{align}\label{right}
	|\nabla \tilde{u} (y_1, y_2)|^2 
	&= \left| \frac{\pd u}{\pd x_1} + \frac{\pd u}{\pd x_1} h' (x_1) \right|^2 + \left| \frac{\pd u}{\pd x_2} \right|^2 \notag \\
	&= | \nabla u (x_1, x_2) |^2 + 2 \frac{\pd u}{\pd x_1} \frac{\pd u}{\pd x_2} h'(x_1) + \left| \frac{\pd u}{\pd x_2} \right|^2 |\,h'(x_1)|^2 \notag \\
	&\le (1+ \delta)^2 | \nabla u (x_1, x_2)|^2.
\end{align}
Also, for any $B \in [0, 1)$
\begin{align}\label{right2}
	\( \log \frac{a}{|y|} \)^B 
	&= \( \log \frac{a}{\sqrt{|x|^2 + |h(x_1)|^2 - 2x_2 h(x_1)}} \)^B \notag \\
	&\le \( \log \frac{a}{\sqrt{|x|^2  - 2 \delta |x_2| |h(x_1)|}} \)^B \notag \\
	&\le \( \log \frac{a}{|x|} \)^B \( 1+ \( \frac{\log \frac{1}{1-2\delta}}{2 \log a} \)^B \) 
	\le \( 1+ C \delta^B \) \( \log \frac{a}{|x|} \)^B
\end{align}
and for any $B \le 0$
\begin{align}\label{right3}
	\( \log \frac{a}{|y|} \)^B 
	\le \( \log \frac{a}{\sqrt{|x|^2  + \delta^2 |x|^2 + 2 \delta |x|^2}} \)^B 
	\le \( \log \frac{a}{|x|} \)^B.
\end{align}
On the other hand, we have
\begin{align}\label{left}
	\int_{B^0_r} \frac{|\tilde{u}|^p}{|y|^2 \( \log \frac{a}{|y|} \)^A} \,dy
	= \int_{B^h_r} \frac{|u|^p}{ \( |x|^2 + |h(x_1)|^2 - 2x_2 h(x_1) \) \( \log \frac{a}{\sqrt{|x|^2 + |h(x_1)|^2 - 2x_2 h(x_1)}} \)^A} \,dx.
\end{align}
Since $h(0)=|h'(0)|=0$ and $|h'(x_1)| \le \delta$ for any $x_1 \in (-r, r)$, we have $|h (x_1)| \le \delta |x_1|$ for any $x_1 \in (-r, r)$. 
Also note that $\( \log \frac{a}{\sqrt{|x|^2 + |h(x_1)|^2 - 2x_2 h(x_1)}} \)^A \ge \( \log \frac{a}{|x|} \)^A$ for any $x_1 \in (-r, r)$. 
Thus we have
\begin{align*}
	&\left| \frac{1}{ \( |x|^2 + |h(x_1)|^2 - 2x_2 |h(x_1)| \) \( \log \frac{a}{\sqrt{|x|^2 + |h(x_1)|^2 - 2x_2 h(x_1)}} \)^A} - \frac{1}{|x|^2 \( \log \frac{a}{|x|} \)^A} \right| \\
	&\le \left| \frac{1}{\( |x|^2 + |h(x_1)|^2 - 2|x_2| |h(x_1)| \)} - \frac{1}{|x|^2} \right| \( \log \frac{a}{|x|} \)^{-A} \\
	&= \frac{(2|x_2| + |h(x_1)| ) |h(x_1)|}{|x|^2(|x|^2-2|x_2||h(x_1)|} \( \log \frac{a}{|x|} \)^{-A} \\
	&\le \frac{(2|x_2| + \delta |x_1| ) |h(x_1)| |x_1|}{|x|^2(|x|^2-2\delta |x_2||x_1|)} \( \log \frac{a}{|x|} \)^{-A} \le \frac{C \delta}{|x|^2} \( \log \frac{a}{|x|} \)^{-A}.
\end{align*}
for any $x \in B_r^h$. 
Therefore, from (\ref{left}), we have
\begin{align}\label{left2}
	\int_{B^0_r} \frac{|\tilde{u}|^p}{|y|^2 \( \log \frac{a}{|y|} \)^A} \,dy
	\ge (1- C\delta) \int_{B^h_r} \frac{|u|^p}{|x|^2 \( \log \frac{a}{|x|} \)^A} \,dx.
\end{align}
From (\ref{tilde u}-\ref{right3}) and (\ref{left2}), if we choose $\delta_0 > 0$ such that $(1+ \delta)^2 (1+ C\delta^B) (1-C \delta)^{-\frac{2}{p}} \le 1+ \ep$ for any $\delta \in (0, \delta_0)$. 
Then we have
\begin{align*}
	2^{\frac{2}{p} -1} C_{p, A, B} \( \int_{B^h_r} \frac{|u|^p}{|x|^2 \( \log \frac{a}{|x|} \)^A} \,dx \)^{\frac{2}{p}} \le (1+ \ep) \int_{B^h_r} \( \log \frac{a}{|x|} \)^B |\nabla u|^2 \,dx.
\end{align*}
\end{proof}

Next is the critical Hardy inequality for functions in $H^1(\Omega)$.

%
%
\begin{prop}\label{prop H^1 CH}
Assume $0 \in \ol{\Omega}$.
For any $\ep > 0$, there exists a constant $C= C(\ep, a, \Omega) > 0$ such that the inequality
\begin{equation}
\label{H^1 CH}
	\int_{\Omega} \frac{u^2}{|x|^2 \( \log \frac{a}{|x|} \)^2} \,dx \le (4+\ep) \int_{\Omega} |\nabla u |^2 \,dx + C \int_{\Omega} u^2 \,dx
\end{equation}
holds for any $u \in H^1 (\Omega)$.
\end{prop}

\begin{proof}

%
%
\noindent 
(I)\hspace{1em} 
Assume that $0 \in \Omega$. 
Then there exists $\delta > 0$ such that $B_{2\delta} \subset \Omega$. 
Let $\phi \in C_c^\infty (\Omega), 0\le \phi \le 1, \phi \equiv 1$ on $B_{\delta}$ and $\phi \equiv 0$ on $\Omega \setminus B_{2\delta}$. 
Then, for any $u \in H^1 (\Omega)$, we have
\begin{align*}
	\int_{\Omega} \frac{u^2}{|x|^2 \( \log \frac{a}{|x|} \)^2} \,dx 
	&= \int_{B_\delta} \frac{u^2}{|x|^2 \( \log \frac{a}{|x|} \)^2} \,dx + \int_{\Omega \setminus B_\delta} \frac{u^2}{|x|^2 \( \log \frac{a}{|x|} \)^2} \,dx  \\
	&\le \int_{B_{2\delta}} \frac{|u \phi |^2}{|x|^2 \( \log \frac{a}{|x|} \)^2} \,dx +  C \int_{\Omega \setminus B_\delta} u^2 \,dx.
\end{align*}
From the critical Hardy inequality \eqref{H^1_0 CH} on $H_0^1(B_{2\delta})$, we have
\begin{align*}
	\int_{\Omega} \frac{u^2}{|x|^2 \( \log \frac{a}{|x|} \)^2} \,dx 
	&\le 4 \int_{B_{2\delta}} |\nabla (u \phi) |^2 \,dx +  C \int_{\Omega} u^2 \,dx  \\
	&\le (4+\ep) \int_{\Omega} |\nabla u |^2 \,dx + C \int_{\Omega} u^2 \,dx,
\end{align*}
which yields the result.

%
%
\noindent 
(II)\hspace{1em} 
Assume that $0 \in \pd \Omega$. 
Since $\pd \Omega$ is smooth, we may assume that $\pd \Omega$ is represented by the graph $x_2 = h(x_1)$, where $h\in C^1 (\re)$ and $h(0) = |h'(0)| = 0$ near the origin. 
Namely, it holds that $\Omega \cap B_r = B^h_r$ for any small $r >0$. 
From Lemma \ref{Lemma ineq log}, for any $\ep > 0$ there exist small $\delta, r > 0$ such that
if $|h'(x_1)| \le \delta$ for any $x_1 \in (-r, r)$, 
then it holds
\begin{align}\label{ep/2}
	&\int_{B^h_r} \frac{|u|^2}{|x|^2 \( \log \frac{a}{|x|} \)^2} \,dx 
	\le \(4 + \frac{\ep}{2} \) \int_{B^h_r} |\nabla u |^2 \,dx\quad {}^{\forall}u \in H^1 (B_r)\,\text{with}\, {\rm supp}\,u \subset B_r.
\end{align}
Let $\phi \in C_c^\infty (B_r)$, $0 \le \phi \le 1$, $\phi \equiv 1$ on $B_{\frac{r}{2}}$ and $\phi \equiv 0$ on $B_r \setminus B_{\frac{r}{2}}$. 
From \eqref{ep/2}, it holds that for any $u \in H^1 (\Omega)$
\begin{align*}
	\int_{\Omega} \frac{u^2}{|x|^2 \( \log \frac{a}{|x|} \)^2} \,dx 
	&= \int_{B^h_{\frac{r}{2}}} \frac{u^2}{|x|^2 \( \log \frac{a}{|x|} \)^2} \,dx + \int_{\Omega \cap (B_r^h)^c } \frac{u^2}{|x|^2 \( \log \frac{a}{|x|} \)^2} \,dx  \\
	&= \int_{B^h_r} \frac{|u\phi|^2}{|x|^2 \( \log \frac{a}{|x|} \)^2} \,dx + C \int_{\Omega} u^2 \,dx  \\
	&\le \(4 + \frac{\ep}{2} \) \int_{B^h_r} |\nabla (u \phi) |^2 \,dx +  C \int_{\Omega} u^2 \,dx  \\
	&\le (4+\ep) \int_{\Omega} |\nabla u |^2 \,dx + C \int_{\Omega} u^2 \,dx,
\end{align*}
which ends the proof in this case.
\end{proof}

Now, we show the critical Hardy inequality for functions in $H^1 (\Omega)$ with average zero conditions. 

%
%
\begin{prop}\label{prop average H}
Assume that $0 \in \ol{\Omega}$, $g \in L^2 \( \Omega, \frac{dx}{|x|^2 \( \log \frac{a}{|x|} \)^2} \)$ with $\int_\Omega \frac{g}{|x|^2 \( \log \frac{a}{|x|} \)^2} \,dx =1$. 
Let $a > 1$.
Then there exists a constant $C= C(a, \Omega) > 0$ such that the inequality
\begin{align*}
	\int_{\Omega} \( u - \int_\Omega \frac{ug}{|x|^2 \( \log \frac{a}{|x|} \)^2} \,dx \)^2 \, \frac{dx}{|x|^2 \( \log \frac{a}{|x|} \)^2} \le C \int_{\Omega} |\nabla u |^2 \,dx
\end{align*}
holds for any $u \in H^1 (\Omega)$.
\end{prop}

\begin{proof}
From Proposition \ref{prop H^1 CH}, we have
\begin{align*}
	&\int_{\Omega} \( u - \int_\Omega \frac{ug}{|x|^2 \( \log \frac{a}{|x|} \)^2} \,dx \)^2 \, \frac{dx}{|x|^2 \( \log \frac{a}{|x|} \)^2} \\
	&\le C \left\|\,u - \int_\Omega \frac{ug}{|x|^2 \( \log \frac{a}{|x|} \)^2} \,dx \,\right\|^2_{H^1 (\Omega)} \\
	&\le C \, \| \nabla u \|^2_{L^2 (\Omega)} + C \,\left\|\,u - \int_\Omega \frac{ug}{|x|^2 \( \log \frac{a}{|x|} \)^2} \,dx \,\right\|^2_{L^2 (\Omega)}.
\end{align*}
Therefore, it is enough to show that there exists a positive constant $C$ such that the inequality
\begin{align*}
	\left\|\,u - \int_\Omega \frac{ug}{|x|^2 \( \log \frac{a}{|x|} \)^2} \,dx \,\right\|^2_{L^2 (\Omega)} \le C\, \| \nabla u \|^2_{L^2 (\Omega)}
\end{align*}
holds for any $u \in H^1 (\Omega)$. Assume that such positive constant $C$ does not exist. Then there exists $\{ u_m \} \subset H^1 (\Omega)$ such that 
\begin{align*}
	\frac{\| \nabla u_m \|_{L^2 (\Omega)}}{\left\|\,u_m - \int_\Omega \frac{u_m \,g}{|x|^2 \( \log \frac{a}{|x|} \)^2} \,dx \,\right\|_{L^2 (\Omega)}} \to 0 \quad (m \to \infty).
\end{align*}
Set 
\begin{align*}
	v_m := \frac{u_m  - \int_\Omega \frac{u_m \,g}{|x|^2 \( \log \frac{a}{|x|} \)^2} \,dx}{\left\|\,u_m - \int_\Omega \frac{u_m \,g}{|x|^2 \( \log \frac{a}{|x|} \)^2} \,dx \,\right\|_{L^2 (\Omega)}}.
\end{align*}
Then we see that 
$\| \nabla v_m \|_{L^2 (\Omega)} \to 0\,(m \to \infty), \, \| v_m \|_{L^2 (\Omega)} =1$ and $\int_\Omega \frac{v_m \,g}{|x|^2 \( \log \frac{a}{|x|} \)^2} \,dx =0$ for any $m \in \N$. 
We may assume that $v_m \rightharpoonup v$ in $H^1 (\Omega)$ and $v_m \to v$ in $L^2 (\Omega)$. 
Then, for any $\phi \in C_c^\infty (\Omega)$, 
\begin{align*}
	\left|  \int_\Omega D_{x_i} v \,\phi \,dx  \right|
	= \left| \int_\Omega v \,\phi_{x_i} \,dx \right|
	&= \left| \lim_{m \to \infty} \int_\Omega v_m \,\phi_{x_i} \,dx \right|\\
	&= \left| \lim_{m \to \infty} \int_\Omega D_{x_i} v_m \,\phi \,dx \right|
	\le \lim_{m \to \infty} \| \nabla v_m \|_{L^2 (\Omega)} \, \| \nabla \phi \|_{L^2 (\Omega)}  =0.
\end{align*}
Therefore, $D v = 0$ a.e. in $\Omega$ which implies $v$ is a constant function. 
Since $\int_\Omega \frac{v \,g}{|x|^2 \( \log \frac{a}{|x|} \)^2} \,dx = \lim_{m \to \infty} \int_\Omega \frac{v_m \,g}{|x|^2 \( \log \frac{a}{|x|} \)^2} \,dx =0$, 
we get $v =0$ which contradicts $\| v \|_{L^2 (\Omega)} = \lim_{m \to \infty} \| v_m \|_{L^2 (\Omega)} =1$.
\end{proof}

%
%
If we substitute $g$ for $\( \int_\Omega \frac{dx}{|x|^2 \( \log \frac{a}{|x|} \)^2} \)^{-1}$ in Proposition \ref{prop average H}, 
then we see that the inequality
\begin{equation}
\label{average CH}
	\int_{\Omega}  \frac{|u|^2}{|x|^2 \( \log \frac{a}{|x|} \)^2} \,dx \le C \int_{\Omega} |\nabla u |^2 \,dx
\end{equation}
holds for any $a>1$ and $u \in H^1 (\Omega)$ with $\int_{\Omega} \frac{u}{|x|^2 \( \log \frac{a}{|x|} \)^2} \,dx = 0$.

%
%
\section{Existence of the second eigenfunction}\label{S Existence}

Recall the minimization problem \eqref{la_a}:
\begin{align*}
	&\la_a = \inf_{u \in \mathcal{A}} \int_{\Omega} |\nabla u|^2\,dx, \\ 
	&\mathcal{A} = \left\{ u \in H^1 (\Omega) \setminus \{ 0\} \,\,\middle| \,\, \int_{\Omega} \frac{|u|^2}{|x|^2 \( \log \frac{a}{|x|} \)^2} \,dx = 1, \int_{\Omega} \frac{u}{|x|^2 \( \log \frac{a}{|x|} \)^2} \,dx = 0\, \right\}. 
\end{align*}
Main goal in this section is the following existence result.

%
%
\begin{theorem}\label{thm attain}
Assume $0 \in \ol{\Omega}$ and $a >1$. 
If $\la_a < \frac{1}{4}$, then $\la_a$ is attained. 
\end{theorem}

\begin{proof}
Let $\{ u_m \}_{m=1}^\infty \subset \mathcal{A}$ be a minimizing sequence for $\la_a$. 
Thus, we have $\int_\Omega |\nabla u_m|^2 \,dx \to \la_a$ as $m \to \infty$, $\int_{\Omega} \frac{u_m}{|x|^2 \( \log \frac{a}{|x|} \)^2} \, dx =0$, 
and $\int_{\Omega} \frac{|u_m |^2}{|x|^2 \( \log \frac{a}{|x|} \)^2} \, dx =1$ for each $m$. 
Since $\frac{1}{|x|^2 \( \log \frac{a}{|x|} \)^2} \ge \min \left\{ \frac{1}{(\log a)^2}, \,\frac{e^2}{a^2}\right\}$ for any $x \in \Omega$, 
we have that $\{ u_m \}$ is bounded in $H^1(\Omega)$.
Thus we may assume that 
\begin{align*}
	\begin{cases}
	&u_m \rightharpoonup u \quad \text{in }  H^1 (\Omega), \\ 
	&u_m \rightharpoonup u \quad \text{in }  L^2 \( \Omega, \frac{dx}{|x|^2 \( \log \frac{a}{|x|} \)^2} \), \\
	&u_m \to u \quad \text{a.e. in } \Omega. 
	\end{cases}
\end{align*}
Therefore, we see that $\int_{\Omega} \frac{u}{|x|^2 \( \log \frac{a}{|x|} \)^2} \, dx =0$ 
and 
$$
	0 \le \int_{\Omega} \frac{|u|^2}{|x|^2 \( \log \frac{a}{|x|} \)^2} \, dx \le \liminf_{m \to \infty} \int_{\Omega} \frac{|u_m|^2}{|x|^2 \( \log \frac{a}{|x|} \)^2} \, dx =1.
$$
We claim that
\begin{align}\label{int 1}
	\int_{\Omega} \frac{|u|^2}{|x|^2 \( \log \frac{a}{|x|} \)^2} \, dx =1.
\end{align}
If we show (\ref{int 1}), then we have
$\la_a \le \int_\Omega |\nabla u|^2 \,dx \le \liminf_{m \to \infty} \int_\Omega |\nabla u_m|^2 \,dx = \la_a$
which implies that $u \in H^1 (\Omega)$ is a minimizer of $\la_a$.

First, we show that $u \not\equiv 0$. Assume the contrary that $u \equiv 0$.
Then by the compactness of the embedding $H^1 (\Omega) \hookrightarrow L^2 (\Omega)$, we see $\int_{\Omega} u_m^2 \, dx \to 0$.
Thus by the critical Hardy inequality for $H^1(\Omega)$ \eqref{H^1 CH}, we have
\begin{align*}
	1 = \int_{\Omega} \frac{|u_m|^2}{|x|^2 \( \log \frac{a}{|x|} \)^2} \, dx
	&\le (4 + \ep) \int_{\Omega} |\nabla u_m|^2 \,dx + o(1) \\
	&\le (4 + \ep) (\la_a + o(1)) + o(1).
\end{align*}
Letting $m \to \infty$ and then $\ep \to 0$, we have $1 \le 4 \la_a$, which contradicts the assumption $\la_a < \frac{1}{4}$.
Therefore, $u \not\equiv 0$. 
Let us assume by contradiction that \eqref{int 1} does not hold and $\int_\Omega \frac{|u|^2}{|x|^2 \( \log \frac{a}{|x|} \)^2} \, dx < 1$. 
Since 
$$
	u_m -u \rightharpoonup 0 \quad \text{in $H^1(\Omega)$} \quad \text{and} \quad \text{in $L^2\(\Omega, \frac{dx}{|x|^2 \( \log \frac{a}{|x|} \)^2}\)$}, 
$$
we see $\int_{\Omega} |u_m - u|^2 \, dx = o(1)$ and by \eqref{H^1 CH}, 
again we have
\begin{align*}
	0< 1- \int_\Omega \frac{|u|^2}{|x|^2 \( \log \frac{a}{|x|} \)^2} \, dx 
	&= \int_\Omega \frac{|u_m|^2}{|x|^2 \( \log \frac{a}{|x|} \)^2} \, dx - \int_\Omega \frac{|u|^2}{|x|^2 \( \log \frac{a}{|x|} \)^2} \, dx \\
	&= \int_\Omega \frac{|u_m -u|^2}{|x|^2 \( \log \frac{a}{|x|} \)^2} \, dx + o(1) \\
	&\le (4 + \ep) \int_{\Omega} |\nabla (u_m -u ) |^2 \,dx + C \int_{\Omega} |u_m - u|^2 \, dx + o(1) \\
	&\le (4 + \ep) \( \int_{\Omega} |\nabla u_m |^2 \,dx - \int_{\Omega} |\nabla u |^2 \,dx \) + o(1) \\
	&= (4 + \ep) \la_a -  (4 + \ep) \int_{\Omega} |\nabla u|^2 \,dx + o(1) 
\end{align*}
as $m \to \infty$.
Also since $\int_{\Omega} \frac{u}{|x|^2 \( \log \frac{a}{|x|} \)^2} \, dx =0$, 
we have $\la_a \int_{\Omega} \frac{|u|^2}{|x|^2 \( \log \frac{a}{|x|} \)^2} \, dx \le \int_{\Omega} |\nabla u|^2 \, dx$
by the definition of $\la_a$. 
Therefore letting $m \to \infty$, we have
\begin{align*}
	1- \int_\Omega \frac{|u|^2}{|x|^2 \( \log \frac{a}{|x|} \)^2} \, dx \le (4+\ep) \la_a \(1- \int_\Omega \frac{|u|^2}{|x|^2 \( \log \frac{a}{|x|} \)^2} \, dx \),
\end{align*}
which implies that $1 \le (4+\ep) \la_a$. 
Again, letting $\ep \to 0$, we have a contradiction to the assumption of $\la_a$.
Therefore, we have (\ref{int 1}).

The proof is now complete.
\end{proof}

In the end of this section, we give two examples of the domain which satisfies $\la_a < \frac{1}{4}$. 

\vspace{1em}
%
%
We say that the domain $\Omega$ satisfies condition (A) if there exist $\delta$, $\theta_*$, $\theta^* >0$, 
$0 \le \theta_* < \theta^* < 2\pi$ such that 
$$
	\{ (r, \theta) \in (1-\delta, 1) \times (\theta_*, \theta^*) \} \subset \Omega
$$
holds, where $(r, \theta)$ is a polar coordinate in $\re^2$.

%
%

\begin{prop}\label{prop value}
If $\Omega$ satisfies condition (A), then there exists a positive constant $C$ such that $\la_a \le C \log a$ for $a >1$. 
\end{prop}

\begin{proof}
Let $\varphi_\delta$ be a smooth cut-off function which satisfies $\varphi_\delta (r) = 1$ for $r \in \left[1-\frac{\delta}{2}, 1 \right]$, 
$\varphi_\delta (r) = 0$ for $r \in [0, 1 -\delta], 0 \le \varphi_\delta \le 1$, and $|\nabla \varphi_\delta| \le C \delta^{-1}$. 
Consider the following test function.
\begin{align*}
	\phi_\delta (x) = \phi_\delta (r, \theta) = 
	\begin{cases}
	r \varphi_\delta (r) \sin \( \frac{ 2\pi (\theta -\theta_* )}{\theta^* - \theta_*} \) \,\,&\text{if}\,\, \theta \in [\theta_*, \theta^*]\\
	0 &\text{if}\,\, \theta \not\in [\theta_*, \theta^*]
	\end{cases}
\end{align*}
Note that 
\begin{align*}
	\int_{\Omega} \frac{\phi_\delta}{|x|^2 \( \log \frac{a}{|x|} \)^2} \,dx 
	= \( \int_{1-\delta}^1 \frac{\varphi_\delta (r)}{\( \log \frac{a}{|x|} \)^2} \,dr\)  \( \int_{\theta_*}^{\theta^*} \sin \( \frac{ 2\pi (\theta -\theta_* )}{\theta^* - \theta_*} \) \,d\theta \) = 0
\end{align*}
Applying $\phi_\delta$ to $\la_a$ implies that
\begin{align*}
	\la_a 
	&\le \frac{\int_{1-\delta}^1 \int_{\theta_*}^{\theta^*} \left[ \left| (r \varphi_\delta )' \right|^2 \sin^2 \( \frac{ 2\pi (\theta -\theta_* )}{\theta^* - \theta_*} \) 
	+ \varphi_\delta^2 \( \frac{ 2\pi}{\theta^* - \theta_*} \)^2 \cos^2 \( \frac{ 2\pi (\theta -\theta_* )}{\theta^* - \theta_*} \)  \right] \, r \,drd\theta}
	{\( \int_{1-\delta/2}^1 \frac{r}{\( \log \frac{a}{r} \)^2} \,dr\)  \( \int_{\theta_*}^{\theta^*} \sin^2 \( \frac{ 2\pi (\theta -\theta_* )}{\theta^* - \theta_*} \) \,d\theta \)}.
\end{align*}
Since
\begin{align*}
	\int_{1-\delta/2}^1 \frac{r}{\( \log \frac{a}{r} \)^2} \,dr \ge \( 1-\frac{\delta}{2} \)^2 \int_{\log a}^{\log a - \log \( 1-\delta /2\)} \frac{dt}{t^2} = O \( (\log a)^{-1} \)
\end{align*}
as $a \to 1$, we have $\la_a \le C \log a$ for $a >1$. 
\end{proof}

%
%
For the next example, we say that a pair $(a, L)$ is {\it admissible} if $a > 1$, $0 < L < 1$, and
\begin{equation*}
	8\pi < \int_{B_L} \frac{1}{\( \log \frac{a}{|x|} \)^2} \,dx.
\end{equation*}
holds.

\begin{lemma}
The set of admissible pairs $(a, L)$ is non-empty.
\end{lemma}

\begin{proof}
Indeed, since $\frac{1}{\( \log \frac{a}{|x|} \)^2} \in L^1(B_1)$ for any $a > 1$, we have
\begin{align*}
	\lim_{L \to 1-0} \int_{B_L} \frac{1}{\( \log \frac{a}{|x|} \)^2} \,dx = \int_{B_1} \frac{1}{\( \log \frac{a}{|x|} \)^2} \,dx.
\end{align*}
So it is enough to show that the following integral
\begin{align*}
	\int_{B_1} \frac{1}{\( \log \frac{a}{|x|} \)^2} \,dx = 2\pi \int_0^1 \frac{r}{\( \log \frac{a}{r} \)^2} dr = 2\pi a^2 \int_{\log a}^{\infty} \frac{e^{-2t}}{t^2} dt
\end{align*}
is large enough for some $a > 1$ and $0 < L < 1$.
Since
\begin{align*}
	2\pi a^2 \int_{\log a}^{\infty} \frac{e^{-2t}}{t^2} dt &= 2\pi a^2 \left\{ \left[ \(-\frac{1}{t}\) e^{-2t} \right]_{\log a}^{\infty} - 2 \int_{\log a}^{\infty} \frac{e^{-2t}}{t} dt \right\} \\
	&= \frac{2\pi}{\log a} - 4\pi a^2 \int_{\log a}^{\infty} \frac{e^{-2t}}{t} dt,
\end{align*}
we have
\begin{align*}
	4\pi a^2 \int_{\log a}^{\infty} \frac{e^{-2t}}{t} dt &= 4\pi a^2 \( \int_{\log a}^{\log(a+1)} \frac{e^{-2t}}{t} dt + \int_{\log (a+1)}^{\infty} \frac{e^{-2t}}{t} dt \) \\ 
	&\le 4\pi a^2 \( \int_{\log a}^{\log(a+1)} \frac{e^{-2\log a}}{t} dt + \int_{\log (a+1)}^{\infty} \frac{e^{-2t}}{\log(a+1)} dt \) \\ 
	&= 4\pi a^2 \( a^{-2} \int_{\log a}^{\log(a+1)} \frac{1}{t} dt + \frac{1}{\log (a+1)} \int_{\log (a+1)}^{\infty} e^{-2t} dt \) \\ 
	&= 4\pi a^2 \( a^{-2} \( \log \log(a+1) - \log \log a \) + \frac{1}{\log (a+1)} \frac{e^{-2\log (a+1)}}{2} \) \\ 
	&= 4\pi \log \log(a+1) - 4 \pi \log \log a + \frac{2\pi}{\log (a+1)} \frac{a^2}{(a+1)^2} \\ 
	&= - 4 \pi \log \log a + O(1) \quad (\text{as} \ a \to 1+0).
\end{align*}
Therefore,
\begin{align*}
	\int_{B_1} \frac{1}{\( \log \frac{a}{|x|} \)^2} \,dx &= \frac{2\pi}{\log a} - 4\pi a^2 \int_{\log a}^{\infty} \frac{e^{-2t}}{t} dt \\
	&\ge \frac{2\pi}{\log a} - \(- 4 \pi \log \log a + O(1) \) \\
	&= \frac{2\pi}{\log a} + 4 \pi \log \log a + O(1) \to +\infty \quad ( a \to 1+0 ).
\end{align*}
Then, there exists $a_0 > 1$ such that  
$\int_{B_1} \frac{1}{\( \log \frac{a_0}{|x|} \)^2} \,dx > 16 \pi$. 
Therefore, there exists $L_0 < 1$ such that 
$\int_{B_{L_0}} \frac{1}{\( \log \frac{a_0}{|x|} \)^2} \,dx > 8 \pi$. 
\end{proof}

%
%
Let $0 < L < 1$. 
We say that the domain $\Omega \subset \re^2$ satisfies the condition $(L)$ if 
\begin{align*}
	B_L(0) \subset \Omega \subset B_1 \quad \text{and} \quad \sum_{i=1}^2 \( \int_{\Omega} \frac{x_i}{|x|^2 \( \log \frac{a}{|x|} \)^2} \,dx \)^2 \ne 0.
\end{align*}

%
%
\begin{prop}\label{prop value2}
If $\Omega$ satisfies the condition $(L)$ for some $L \in (0, 1)$ and $(a, L)$ is admissible, 
then $\la_a(\Omega) < 1/4$. 
\end{prop}

\begin{proof}
Set $\alpha_i = \int_{\Omega} \frac{x_i}{|x|^2 \( \log \frac{a}{|x|} \)^2} \,dx$ for $i=1,2$ and $u(x) = \alpha_2 x_1 - \alpha_1 x_2$. 
Note that $u \in H^1 (\Omega)$ satisfies that $\int_{\Omega} \frac{u}{|x|^2 \( \log \frac{a}{|x|} \)^2} \,dx = 0$. 
Testing $\la_a (\Omega)$ by $u(x) = \al_2 x_1 - \al_1 x_2$, we have
\begin{align*}
	\int_{\Omega} |\nabla u|^2 dx = (\al_1^2 + \al_2^2) |\Omega|.
\end{align*}
From the symmetry, we have
\begin{align*}
	&\int_{\Omega} \frac{|u|^2}{|x|^2 \( \log \frac{a}{|x|} \)^2} \,dx \ge \int_{B_L} \frac{|u|^2}{|x|^2 \( \log \frac{a}{|x|} \)^2} \,dx \\
	&= \al_1^2 \int_{\Omega} \frac{x_1^2}{|x|^2 \( \log \frac{a}{|x|} \)^2} \,dx + \al_2^2 \int_{B_L} \frac{x_2^2}{|x|^2 \( \log \frac{a}{|x|} \)^2} \,dx 
	+ 2 \al_1 \al_2 \int_{B_L} \frac{x_1 x_2}{|x|^2 \( \log \frac{a}{|x|} \)^2} \,dx \\
	&= \frac{\al_1^2 + \al_2^2}{2} \int_{B_L} \frac{1}{\( \log \frac{a}{|x|} \)^2} \,dx.
\end{align*}
Therefore, 
\begin{align*}
	\la_a(\Omega) &\le \frac{\int_{\Omega} |\nabla u|^2\,dx }{\int_{\Omega} \frac{|u|^2}{|x|^2 \( \log \frac{a}{|x|} \)^2} \,dx} 
	\le \frac{(\al_1^2 + \al_2^2) |\Omega|}{\frac{\al_1^2 + \al_2^2}{2} \int_{B_L} \frac{1}{\( \log \frac{a}{|x|} \)^2} \,dx} 
	= \frac{2|\Omega|}{\int_{B_L} \frac{1}{\( \log \frac{a}{|x|} \)^2} \,dx} \\
	&\le \frac{2\pi}{\int_{B_L} \frac{1}{\( \log \frac{a}{|x|} \)^2} \,dx} < \frac{1}{4}
\end{align*}
since $|\Omega | < |B_1| = \pi$ and $(a, L)$ is admissible. 
Thus we have $\la_a(\Omega) < 1/4$. 
\end{proof}

From Proposition \ref{prop value}, Proposition \ref{prop value2} and Theorem \ref{thm attain}, we have the following.


\begin{cor}\label{cor attain}
Assume $\Omega$ satisfies the condition (A) and $a >1$ sufficiently close to 1,
or condition (L) for an admissible pair $(a, L)$.
Then there exists the second eigenfunction of the eigenvalue problem (N).
\end{cor}

\begin{proof}
From Proposition \ref{prop value} or Proposition \ref{prop value2}, we have $\la_a < 1/4$ and by Theorem \ref{thm attain}, 
there exists a minimizer $u$ of $\la_a$. 
Now, we check the Euler-Lagrange equation satisfied by $u$.
Let $\phi \in H^1 (\Omega)$ with $\int_\Omega \frac{\phi }{|x|^2 \( \log \frac{a}{|x|} \)^2} \, dx =0$. 
Set
\begin{align*}
	h (t) = \frac{\int_{\Omega} |\nabla (u + t \phi )|^2\,dx }{\int_{\Omega} \frac{| u + t \phi |^2}{|x|^2 \( \log \frac{a}{|x|} \)^2} \,dx}. 
\end{align*}
Since the function $h$ attains a minimum at $t =0$, we have
\begin{align*}
	\int_\Omega \nabla u \cdot \nabla \phi \,dx - \la_a \int_\Omega \frac{u \phi }{|x|^2 \( \log \frac{a}{|x|} \)^2} \, dx =0.
\end{align*}
To extend this identity for any $\phi \in H^1 (\Omega),$ we set 
\begin{align*}
	\psi = \phi - \frac{\int_\Omega \frac{\phi }{|x|^2 \( \log \frac{a}{|x|} \)^2} \, dx}{\int_\Omega \frac{1}{|x|^2 \( \log \frac{a}{|x|} \)^2} \, dx}.
\end{align*}
Since $\int_\Omega \frac{\psi }{|x|^2 \( \log \frac{a}{|x|} \)^2} \, dx=0$ and $u$ is orthogonal to $1$ in $L^2 \( \Omega, \frac{dx}{|x|^2 \( \log \frac{a}{|x|} \)^2} \)$, we have
\begin{align*}
	\int_\Omega \nabla u \cdot \nabla \phi \,dx
	= \int_\Omega \nabla u \cdot \nabla \psi \,dx
	=\la_a \int_\Omega \frac{u \psi }{|x|^2 \( \log \frac{a}{|x|} \)^2} \, dx
	= \la_a \int_\Omega \frac{u \phi }{|x|^2 \( \log \frac{a}{|x|} \)^2} \, dx
\end{align*}
for any $\phi \in H^1 (\Omega)$. 
\end{proof}

%
%
\section{Asymptotic behavior of the second eigenfunction around $0$.}\label{S Behavior}

In this section, we study the asymptotic behavior of the second eigenfunction around the origin, which is obtained in \S \ref{S Existence}. 
To obtain the asymptotic behavior in our case by De Giorgi-Nash-Moser iteration technique (\cite{Han}, \cite{HL}, \cite{CPR}),
we need a Sobolev type inequality with a logarithmic weight in Theorem \ref{Thm ineq log} and its extension to $H^1(B_r)$ in Lemma \ref{Lemma ineq log}.

\begin{theorem}\label{Thm singularity}
Let $\la_a \in (0, \frac{1}{4})$ be the second eigenvalue of $(N)$ and $u_a \in H^1 (\Omega)$ be the corresponding second eigenfunction of $(N)$. 
Then there exist $\delta > 0$ and $C >0$ such that
\begin{align*}
	|u_a (x)| \le C \( \log \frac{a}{|x|} \)^{\frac{1}{2} - \frac{\sqrt{1- 4 \la_a}}{2}}\quad 
	\begin{cases}
	&\text{for}\,\, x \in B_\delta \setminus \{ 0\} \quad \text{if} \,\, 0 \in \Omega, \\
	&\text{for}\,\, x \in B_\delta \cap \Omega \quad \text{if} \,\, 0 \in \pd\Omega.
	\end{cases}
\end{align*}
\end{theorem}


\begin{remark}
It is known that for $\nu \in (0,1)$, $a \ge e$, and $0 \in \Omega$,
the solution $u_\nu \in H_0^1 (\Omega)$ of the following eigenvalue problem corresponding $\la = \la(\nu)$:
\begin{align*}
(D)\,
	\begin{cases}
	-\lap u - \frac{\nu}{4} \frac{u}{|x|^2 \( \log \frac{a}{|x|} \)^2} = \la u \quad \text{in}\,\, \Omega, \\
	u>0 \quad \text{in}\,\, \Omega,\, u = 0 \quad \text{on} \,\, \pd \Omega.
	\end{cases}
\end{align*}
satisfies the following estimates.
\begin{align*}
	&C_1 \le \liminf_{x \to 0} \( \log \frac{a}{|x|} \)^{-\frac{1}{2} + \frac{\sqrt{1- \nu}}{2}} |u_\nu (x)| \le \limsup_{x \to 0} \( \log \frac{a}{|x|} \)^{-\frac{1}{2} + \frac{\sqrt{1- \nu}}{2}} |u_\nu (x)| \le C_2 \\
	&\limsup_{x \to 0} \( \log \frac{a}{|x|} \)^{-\frac{1}{2} + \frac{\sqrt{1- \nu}}{2}} |x| \,|\nabla u_\nu (x)| \le C_2
\end{align*}
See \cite{AS} Theorem 1.5. 
\end{remark}

\begin{proof}
Put $u_a = u$ for simplicity. Define 
\begin{align*}
	v(x) = \( \log \frac{a}{|x|} \)^{-\alpha} u(x) = \frac{u(x)}{V_a(x)},
\end{align*}
where $V_a(x) =  \( \log \frac{a}{|x|} \)^{\alpha}$ and $\alpha = \frac{1}{2} - \frac{\sqrt{1- 4 \la_a}}{2}$. 
Note that $V_a$ satisfies the equation $-\Delta V_a =  \frac{\la_a}{|x|^2 \(\log \frac{a}{|x|}\)^2} V_a$ in $\Omega$.
By straightforward calculations, we see that $v \in H^1 \( \Omega, \( \log \frac{a}{|x|} \)^{2\alpha } dx \)$ and 
\begin{align}\label{eq v}
	{\rm div} \( \( \log \frac{a}{|x|} \)^{2\alpha } \nabla v \) = 0 \quad \text{in}\,\, \Omega. 
\end{align} 
We shall show $v \in L^\infty$.
Let $0 < r < \rho$. We put $\phi = \eta^2 v v^{2(s-1)}_\ell$, where $\ell, s >1$, $v_\ell = \min\{ |v|, \ell \}$, 
and $\eta$ is a $C^1$-function such that $\eta = 1$ on $B_r, \eta =0$ on $\re^2 \setminus B_\rho$ and $|\nabla \eta| \le \frac{4}{\rho -r}$ on $\re^2$. Testing (\ref{eq v}) with $\phi$ we have
\begin{align*}
	0&= \int_{\Omega} \( \log \frac{a}{|x|} \)^{2\alpha} \nabla v \cdot \nabla \( \eta^2 v v^{2(s-1)}_\ell \) \,dx \\
	&= \int_{\Omega} \( \log \frac{a}{|x|} \)^{2\alpha} \left[ 2\eta v v^{2(s-1)}_\ell \nabla v \cdot \nabla \eta + \eta^2 v^{2(s-1)}_\ell |\nabla v |^2 + 2(s-1) v^{2(s-1)}_\ell |\nabla v_\ell |^2 \eta^2 \right] \,dx.
\end{align*}
Therefore, we have
\begin{align*}
	&\int_{\Omega} \( \log \frac{a}{|x|} \)^{2\alpha} \left[ \eta^2 v^{2(s-1)}_\ell |\nabla v |^2 + 2(s-1) v^{2(s-1)}_\ell |\nabla v_\ell |^2 \eta^2 \right] \,dx \\
	&= - \int_{\Omega} \( \log \frac{a}{|x|} \)^{2\alpha} 2\eta v v^{2(s-1)}_\ell \nabla v \cdot \nabla \eta \,dx \\
	&\le \frac{1}{2} \int_{\Omega} \( \log \frac{a}{|x|} \)^{2\alpha}  \eta^2 v^{2(s-1)}_\ell |\nabla v |^2 \,dx + C \int_{\Omega} \( \log \frac{a}{|x|} \)^{2\alpha} v^2 v^{2(s-1)}_\ell | \nabla \eta |^2 \,dx
\end{align*}
which implies that
\begin{align}\label{estimate1}
	&\int_{\Omega} \( \log \frac{a}{|x|} \)^{2\alpha} \left[ \frac{1}{2} \eta^2 v^{2(s-1)}_\ell |\nabla v |^2 + 2(s-1) v^{2(s-1)}_\ell |\nabla v_\ell |^2 \eta^2 \right] \,dx  \notag \\
	&\le C \int_{\Omega} \( \log \frac{a}{|x|} \)^{2\alpha} v^2 v^{2(s-1)}_\ell | \nabla \eta |^2 \,dx
\end{align}
Here, we have used the inequality \eqref{ineq log} in Theorem \ref{Thm ineq log} with $u = \eta v v_\ell^{s-1}$. 
We choose $B = 2\alpha \in [0, 1)$, $A >1$, $p= 2\frac{A-1}{1-B} >2$. 
By using Theorem \ref{Thm ineq log} for the case where $0 \in \Omega$ and Lemma \ref{Lemma ineq log} for the case where $0 \in \pd \Omega$, 
we have
\begin{align*}
	&\( \int_{\Omega \cap B_\rho} \frac{|\eta v v_\ell^{s-1}|^p}{|x|^2 (\log \frac{a}{|x|})^A} \,dx \)^{\frac{2}{p}} \\
	&\le C \int_{\Omega \cap B_\rho} \( \log \frac{a}{|x|} \)^{2\alpha} |\nabla (\eta v v_\ell^{s-1})|^2 \,dx\\
	&\le C \int_{\Omega \cap B_\rho} \( \log \frac{a}{|x|} \)^{2\alpha} \left[ |\nabla \eta|^2 v^2 v_\ell^{2(s-1)} + \eta^2 v^{2(s-1)}_\ell |\nabla v |^2 + (s-1)^2 v^{2(s-1)}_\ell |\nabla v_\ell |^2 \eta^2 \right] \,dx\\
	&\le C s \int_{\Omega \cap B_\rho} \( \log \frac{a}{|x|} \)^{2\alpha} v^2 v^{2(s-1)}_\ell | \nabla \eta |^2 \,dx,
\end{align*}
where the last inequality comes from (\ref{estimate1}). 
Since $v^2 v_\ell^{-2} \le v^p v_\ell^{-p}$ and $|x|^2 \( \log \frac{a}{|x|} \)^{A+2\alpha} \le C$ for any $x \in \Omega$, 
we have
\begin{align}\label{estimate2}
	\( \int_{\Omega \cap B_r} \frac{v^2 v_\ell^{ps-2}}{|x|^2 (\log \frac{a}{|x|})^A} \,dx \)^{\frac{2}{p}} 
	\le \frac{Cs}{(\rho -r)^2} \int_{\Omega \cap B_\rho} \frac{v^2 v^{2s-2}_\ell }{|x|^2 (\log \frac{a}{|x|})^A } \,dx.
\end{align}
Take $\rho_0 > 0$ such that $B_{2\rho_0} \subset \Omega$ and
$$
	s_0 =2, s_j = s_0 \( \frac{p}{2} \)^j, r_j = \rho_0 (1+ \rho_0^j) \quad \text{for} \ j =0,1, 2, \cdots. 
$$
Applying the inequality (\ref{estimate2}) with $\rho = r_j, r= r_{j+1}$, 
we obtain
\begin{align*}
	\( \int_{\Omega \cap B_{r_{j+1}}} \frac{ v^2 v_\ell^{2s_{j+1}-2}}{|x|^2 (\log \frac{a}{|x|})^A} \,dx \)^{\frac{1}{2s_{j+1}}} 
	\le \( \frac{Cs_j}{(\rho_0 -\rho_0^2)^2 \rho_0^{2j}} \)^{\frac{1}{2s_j}} \(  \int_{\Omega \cap B_{r_j}} \frac{v^2 v^{2s_j-2}_\ell}{|x|^2 (\log \frac{a}{|x|})^A } \,dx\)^{\frac{1}{2s_j}}.
\end{align*}
Therefore, we have
\begin{align*}
	\( \int_{\Omega \cap B_{r_{j+1}}}v^2 v_\ell^{2s_{j+1}-2} \,dx \)^{\frac{1}{2s_{j+1}}} 
	&\le \( \frac{Cs_j}{\rho_0^{2j}} \)^{\frac{1}{2s_j}} \(  \int_{\Omega \cap B_{r_j}}  v^2 v^{2s_j-2}_\ell \,dx\)^{\frac{1}{2s_j}} \\
	&\le \( \sqrt{C} \)^{\sum_{k=0}^j s_k^{-1}} \rho_0^{- \sum_{k=0}^j k s_k^{-1}} \( \prod_{k=0}^j s_k^{\frac{1}{2 s_k}} \) \(  \int_{\Omega \cap B_{r_0}}  \frac{v^2 v^{2s_0-2}_\ell}{|x|^2 (\log \frac{a}{|x|})^A } \,dx\)^{\frac{1}{2s_0}}
\end{align*}
which implies that
\begin{align*}
	\| v \|_{L^\infty (\Omega \cap B_{\rho_0})} 
	&= \lim_{\ell \to \infty} \lim_{j \to \infty} \| v_\ell \|_{L^{2s_{j+1}} (\Omega \cap B_{\rho_0})} \\
	&\le \( \sqrt{C} \)^{\sum_{k=0}^\infty s_k^{-1}} \rho_0^{- \sum_{k=0}^\infty k s_k^{-1}} \( \prod_{k=0}^\infty s_k^{\frac{1}{2 s_k}} \) \( \int_\Omega \( \log \frac{a}{|x|} \)^{2\alpha} |\nabla v |^2 \,dx\)^{\frac{1}{2}}.
\end{align*}
Since the infinite sums and the infinite product on the right-hand side of the above inequality are finite, $v \in L^\infty (\Omega \cap B_{\rho_0})$. 
The desired result follows.  
\end{proof}

\section{The Robin boundary conditions}

As in the former sections, we can consider the minimization problem 
\begin{align*}
	\la_a^R  := \inf \left\{ \, \frac{\int_{\Omega} |\nabla u|^2 \,dx + \int_{\pd \Omega} \beta u^2 dS }{\int_{\Omega} \frac{|u|^2}{|x|^2 (\log \frac{a}{|x|})^2} \,dx } 
	\,\,\middle| \,\, u \in H^1 (\Omega ) \setminus \{ 0\} \, \right\},
\end{align*}
where $\beta$ is a continuous function on $\pd \Omega$ and $a >1$. 
Then $\la_a^R$ is the first (smallest) eigenvalue of the following linear eigenvalue problem with Robin boundary conditions
\begin{align*}
	(R)\,
	\begin{cases}
	-\lap u = \la \frac{u}{|x|^2 \( \log \frac{a}{|x|} \)^2}\quad &\text{in}\,\, \Omega, \\
	\frac{\pd u}{\pd \nu} + \beta u = 0 \quad &\text{on} \,\, \pd \Omega.
\end{cases}
\end{align*}
Clearly we have $\la_a^R \le \frac{1}{4}$ for any $\beta \in C(\pd\Omega)$ by using $H_0^1 (\Omega) \subset H^1 (\Omega)$ and the inequality (\ref{H^1_0 CH}).
In this section we prove some results about $\la_a^R$.
Since the arguments we use are similar to those in \S \ref{S Existence} and \S \ref{S Behavior}, 
we will omit the most proofs here. 

%
%
First, we assume that $\beta$ is a non-negative continuous function. 
\begin{theorem}\label{thm attain beta +}($\la_a^R$ with a non-negative coefficient $\beta$)
Assume $0 \in \ol{\Omega}$.
Let $a > 1$, $\beta \ge 0$, $\beta \not\equiv 0$. 
Then $\la_a^R >0$. 
Furthermore, if $\la^R_a < \frac{1}{4}$, then $\la^R_a$ is attained. 
\end{theorem}

\begin{remark}
Let $\Omega = B_R (0)$. 
Using the test function $u \equiv 1$, we have
\begin{align*}
	\la_a^R \le \frac{\int_{\pd \Omega} \beta \,dS}{\int_{\Omega} \frac{dx}{|x|^2 (\log \frac{a}{|x|})^2} } \le \| \beta \|_{L^{\infty}(\pd\Omega)} R \log \frac{a}{R}. 
\end{align*}
Therefore, $\la^R_a < \frac{1}{4}$ if $\| \beta \|_{L^{\infty}(\pd\Omega)} R \log \frac{a}{R} < \frac{1}{4}$ holds. 
\end{remark}

\begin{proof}
First we show that $\la_a^R >0$. Obviously $\la_a^R \ge 0$. 
Assume that $\la_a^R =0$.
Then there exists a sequence $\{ u_m \} \subset H^1 (\Omega)$ such that 
\begin{align*}
	\int_{\Omega} |\nabla u_m |^2 \,dx + \int_{\pd \Omega} \beta u_m^2 \,dS \to 0 \quad (m \to \infty), \quad 
	\int_{\Omega} \frac{|u_m|^2}{|x|^2 (\log \frac{a}{|x|})^2} \,dx =1 \quad (\forall m \in \N). 
\end{align*}
Since $\beta \ge 0$, we have $\| \nabla u_m \|_{L^2(\Omega)} = o(1)$ as $m \to \infty$ and moreover $\(\frac{1}{\log a}\)^2 \int_{\Omega} u_m^2 \le \int_{\Omega} \frac{|u_m|^2}{|x|^2 (\log \frac{a}{|x|})^2} \,dx = 1$.
Thus $\{ u_m \}$ is bounded in $H^1(\Omega)$ and we may assume that $u_m \rightharpoonup u$ in $H^1 (\Omega)$ for some $u \in H^1(\Omega)$, $u_m \to u$ in $L^2(\Omega)$ and $L^2(\pd \Omega)$, 
since the embedding $H^1(\Omega) \hookrightarrow L^2(\pd\Omega)$ is compact. 
Since $\nabla u_m \to 0$ in $L^2(\Omega, \re^2)$, we have $\nabla u = 0$, that is $u$ is a constant.
On the other hand, $\int_{\pd\Omega} \beta u_m^2 dS = o(1)$ implies $\int_{\pd\Omega} \beta u^2 dS = 0$, which is impossible if $u$ is a non zero constant.
Thus we have $u \equiv 0$ and $u_m \to 0$ in $H^1(\Omega)$ strongly. 
However this contradicts $\int_{\Omega} \frac{|u_m|^2}{|x|^2 (\log \frac{a}{|x|})^2} \,dx =1$ and the critical Hardy inequality \eqref{H^1 CH} for $u_m \in H^1(\Omega)$. 
Hence $\la_a^R >0$. 

The attainability of $\la_a^R$ can be shown by a minor modification of the proof of Theorem \ref{thm attain}. 
We omit it here.
\end{proof}


\begin{theorem}\label{thm simple}(Simplicity of $\la_a^R$)
Let $a >1,\, \beta \not\equiv 0, \,\beta \ge 0$, and $\la^R_a < \frac{1}{4}$. 
Then the first eigenvalue $\la_a^R$ is simple and the corresponding eigenfunction does not change its sign. 
\end{theorem}

\begin{proof}
The proof follows from \cite{CPR}. We omit it here. 
\end{proof}

\begin{theorem}\label{thm lower upper}(Asymptotic behavior of the first eigenfunction of $\la_a^R$)
Let $a >1$, $\beta \ge 0$, $\beta \not\equiv 0$, $\la^R_a < \frac{1}{4}$, and $u_a$ be a positive minimizer of $\la_a^R$. 
Then there exist a $\delta > 0$ and positive constants $C_1, C_2$ such that 
\begin{align*}
	\begin{cases}
	C_1 \( \log \frac{a}{|x|} \)^{\frac{1}{2} - \frac{\sqrt{1- 4 \la^R_a}}{2}} \le
	&u_a (x) \le C_2 \( \log \frac{a}{|x|} \)^{\frac{1}{2} - \frac{\sqrt{1- 4 \la^R_a}}{2}}
	\quad \text{for}\,\, x \in B_\delta \setminus \{ 0\} \quad \text{if} \ 0 \in \Omega, \\
	&u_a (x) \le C_2 \( \log \frac{a}{|x|} \)^{\frac{1}{2} - \frac{\sqrt{1- 4 \la^R_a}}{2}}
	\quad \text{for}\,\, x \in B_\delta \cap \Omega \quad \text{if} \ 0 \in \pd\Omega,
	\end{cases}
\end{align*}
hold true.
\end{theorem}

\begin{proof}
The upper bound follows from the same argument as in the proof of Theorem \ref{Thm singularity}. 
The lower bound in the case $0 \in \Omega$ follows from the same argument as in the proof of \cite{AS} Theorem 1.5. 
\end{proof}


Next we assume that $\beta$ is a non-positive continuous function. 


\begin{theorem}\label{thm attain beta -}($\la_a^R$ with a non-positive coefficient $\beta$)
Assume $0 \in \ol{\Omega}$.
Let $a >1,\, \beta \not\equiv 0$ and $\beta \le 0$. Then $-\infty < \la_a^R < 0$. 
Furthermore, $\la^R_a$ is attained. 
\end{theorem}

\begin{proof}
Obviously, we see that $\la_a^R <0$ by using the test function $u \equiv 1$. 
Recall the following trace inequality in $H^1(\Omega)$ (\cite{AB} Lemma 1): 
for any $\ep >0$ there exist a constant $C(\ep) >0$ such that 
\begin{align*}
	\int_{\pd \Omega} u^2 \, dS \le \ep \int_{\Omega} |\nabla u|^2 \,dx + C(\ep) \int_{\Omega} u^2 \,dx
\end{align*}
for any $u \in H^1(\Omega)$. 
If we choose $\ep >0$ so that $\ep \,\| \beta \|_{L^\infty(\pd\Omega)} < 1$, 
then the above inequality implies
\begin{align*}
	&\int_{\Omega} |\nabla u|^2 \,dx + \int_{\pd \Omega} \beta u^2 \, dS\\
	&\ge (1- \ep \,\| \beta \|_\infty )\int_{\Omega} |\nabla u|^2 \,dx + \| \beta \|_\infty \( \ep \int_{\Omega} |\nabla u|^2 \,dx - \int_{\pd \Omega}  u^2 \, dS \) \\
	&\ge - \| \beta \|_\infty C(\ep)  \int_{\Omega}  u^2 \, dx
	\ge - \| \beta \|_\infty C \int_{\Omega} \frac{|u|^2}{|x|^2 (\log \frac{a}{|x|})^2} \,dx
\end{align*}
which implies that $\la_a^R$ is bounded from below. 
The attainability of $\la_a^R$ can be shown by a minor modification of the proof of Theorem \ref{thm attain}. 
\end{proof}

Finally, we assume that $\beta$ is a sign-changing continuous function. 
For $\beta$, set $\beta^{\pm} =$max$\{ \,\pm \beta, \,0\}$.
Then $\beta = \beta^+ - \beta^-$.

\begin{theorem}\label{thm attain beta +-}($\la_a^R$ with a sign-changing coefficient $\beta$)
Assume $0 \in \ol{\Omega}$.
Let $a >1,\, \beta^+ \not\equiv 0$ and $\beta^- \not\equiv 0$. 
Then $\la_a^R $ is bounded from below. 
Especially, in the case where $\int_{\pd \Omega} \beta \, dS \le 0$, $\la^R_a < 0$ is attained. 
On the other hand, in the case where $\int_{\pd \Omega} \beta \,dS > 0$, $\la^R_a $ is attained if $\la_a^R < \frac{1}{4}$.
\end{theorem}

\begin{proof}
In the case where $\int_{\pd \Omega} \beta \, dS < 0$, we easily see that $\la_a^R < 0$ by using the test function $u \equiv 1$. 
If $\int_{\pd \Omega} \beta \, dS  =0$, then we test $\la_a^R$ with a function $u(x) = v(x) + t$, 
where $t >0$ is a constant 
and $v \in C_c^1( \overline{\Omega})$ with $v \ge 0$ ${\rm supp } v \,\cap {\rm supp }\beta^- =\emptyset$, and $\int_{\pd \Omega} \beta^- v \,dS > 0$. 
Then we have
\begin{align*}
	\int_{\Omega} |\nabla u |^2 \,dx + \int_{\pd \Omega} \beta u^2 \,dS
	= \int_{\Omega} |\nabla v |^2 \,dx - \int_{\pd \Omega} \beta^- v^2  \,dS - 2t \int_{\pd \Omega} \beta^- v  \,dS < 0
\end{align*}
for sufficiently large $t >0$. 
Therefore, we see that $\la_a^R < 0$. 
The boundedness of $\la_a^R$ from below follows from the proof of Theorem \ref{thm attain beta -}. 
The attainability of $\la_a^R$ can be shown by a minor modification of the proof of Theorem \ref{thm attain}.
\end{proof}

%
%
\section{Sobolev type inequalities with logarithmic weights}\label{S Sob log}

In this section, we prove Theorem \ref{Thm ineq log}. 

Our basic tools are the generalized rearrangement of functions by Talenti \cite{Ta1, Ta2} and the weighted nonlinear potential theory by D. R. Adams \cite{Adams}.
Basically, we follow the arguments by Horiuchi and Kumlin \cite{HK} to establish Theorem \ref{Thm ineq log}.
In \cite{HK}, Caffarelli-Kohn-Nirenberg inequalities with critical or supercritical power type weights are studied.

In the case where $p=2$ or $B=0$, the optimal constant and the attainability of the inequality (\ref{ineq log}) are studied by \cite{HK,S(JDE),RS,N,S(MJM)}. 
Therefore, we shall show the inequality (\ref{ineq log}) in the case where $B \not= 0$ and $p >2$.  

For Trudinger-Moser inequalities with logarithmic weights, see \cite{CR,R}. 
%
%

Recall the following Theorem.
 

\begin{ThmA}(\cite{Adams} Theorem 7.1)
Let $p>2$. Assume that $\w$ belongs to Muckenhoupt $A_2$-class, $g \in L^1_{\rm loc} (\re^2)$ and $g \ge 0$ a.e. on $\re^2$. 
Then the following two assertions are equivalent to each other:
\begin{align*}
	&(1) \sup_{x \in \re^2, \,r >0} \( \int_{B_r (x)} g(x) \,dx \) J[\w](x,r)^{\frac{p}{2}} < \infty, \\
	&(2) \,\text{There exists a constant} \,C >0\, \text{such that for any}\, f \in L^2 (\re^2  ; \,\w \,dx), \\
	&\hspace{2em}\| I_1 * f \|_{L^p (\re^2; \,g(x)\,dx)} \le C \| f \|_{L^2 (\re^2  ; \,\w \,dx)}.
\end{align*}
Here, $I_s (x) = \dfrac{\Gamma \(  \frac{2-s}{2} \)}{2^s \pi \,\Gamma \(  \frac{s}{2} \)} |x|^{-(2-s)}$ is the Riesz potential for $s \in (0,2)$ in $\re^2$ and 
\begin{align*}
J[\w] (x,r) = \int_r^\infty \frac{1}{\pi t^2} \( \int_{B_t (x)} \frac{dy}{\w (y)} \) \, \frac{dt}{t}
\end{align*}
\end{ThmA}

%
%
\vspace{1em}\noindent
{\it Proof of Theorem \ref{Thm ineq log}}:
Let $B \not= 0$ and $p >2$. We divide the proof into two cases with respect to the range of $B$. 

%
%
\noindent
{\bf (I)}\hspace{1em} Let $B < 0$. 
In this case, we can apply a theory of generalized rearrangement of functions (\cite{Ta1,Ta2}). 
Set 
\begin{align*}
f(x) = 
\begin{cases}
	\( \log \frac{a}{|x|} \)^{-B} \quad &\text{if}\,\, |x| \le 1,\\
	\( \log a \)^{-B}   &\text{if}\,\, |x| \ge 1.
\end{cases}
\end{align*}
Since $f \in L^1_{\rm loc} (\re^2) \cap C(\re^2 \setminus \{ 0\})$ is radial, non-increasing with respect to $r = |x|$, and $f \ge 0$ on $\re^2 \setminus \{ 0\}$, $f$ is admissible. Thanks to zero extension, it is enough to show the inequality \eqref{ineq log} for $u \in C_c^\infty (B_1)$. 
Now, we define the rearrangement function $\mathcal{R}_f [u]$ of $u$ with respect to $f$ as follows:
For $x \in \re^2 \setminus \{ 0\}$, 
\begin{align*}
	\mathcal{R}_f [u] (x) &= \mathcal{R}_f [u] (|x|) =\sup \{ \,t \ge 0 \,|\,\, \mu_f [u] (t) > \mu_f \( B_{|x|}  \) \, \},\\
	\mu_f (A) &= \int_A f (x) \,dx, \\
	\mu_f [u] (t) &= \mu_f \( \,\{ \, |u| >t \, \}  \,\) = 
	\int_{\{ \,|u| > t \, \}} f(x) \,dx.
\end{align*}
In the case where $f \equiv 1$, $\mathcal{R}_f [u]$ coincides with the well-known rearrangement $u^{\#}$. 
Given $A, B$, we fix $\tilde{a} > 1$ such that $\frac{1}{|x|^2 \( \log \frac{\tilde{a}}{|x|} \)^{A-B}}$ is decreasing with respect to $r= |x|$.
Then we have
\begin{align*}
	\int_{B_1} \frac{|u|^p}{ |x|^2 \( \log \frac{a}{|x|} \)^A} \,dx 
	&\le C \int_{B_1} \frac{|u|^p}{ |x|^2 \( \log \frac{\tilde{a}}{|x|} \)^{A-B}} f(x) \,dx \\
	&\le C \int_{B_1} \frac{|\, \mathcal{R}_f [u] \,|^p}{ |x|^2 \( \log \frac{\tilde{a}}{|x|} \)^{A-B}} f(x) \,dx 
	\le C \int_{B_1} \frac{|\, \mathcal{R}_f [u] \,|^p}{ |x|^2 \( \log \frac{a}{|x|} \)^A} \,dx,
\end{align*}
where 
the second inequality comes from the Hardy-Littlewood type inequality, see e.g. \cite{HK} Proposition 4.4.
On the other hand, the P\'olya-Szeg\"o type inequality, see e.g. \cite{HK}  Proposition 4.5, implies 
\begin{align*}
	\int_{B_1} \( \log \frac{a}{|x|} \)^{B} |\nabla u|^2 \,dx 
	&= \int_{B_1} \frac{ |\nabla u|^2}{f(x)} \,dx \\
	&\ge \int_{B_1} \frac{ |\nabla \mathcal{R}_f [u] \,|^2}{f(x)} \,dx 
	= \int_{B_1} \( \log \frac{a}{|x|} \)^{B} |\nabla \mathcal{R}_f [u] \,|^2 \,dx.
\end{align*}
Therefore it is enough to show the inequality (\ref{ineq log}) only for radial functions.  
Here, we recall the Hardy type inequality with logarithmic weigts (see e.g. \cite{N,RS} or  \cite{S(MJM)} Corollary 1.3): 
\begin{align}\label{Hardy with log}
\( \frac{1-B}{2} \)^2 \int_{B_1} \frac{|u|^2}{|x|^2 \( \log \frac{a}{|x|} \)^{2-B} }\,dx 
\le \int_{B_1} \frac{|\nabla u|^2}{ \( \log \frac{a}{|x|} \)^{-B} }\,dx\quad (u \in C_c^\infty (B_1))
\end{align}
and the radial lemma with logarithmic weights (see e.g. \cite{CR} lemma 5):
\begin{align}\label{radial lemma with log}
|u(x)| \le \frac{C}{\sqrt{1-B}} \,\( \int_{B_1} |\nabla u|^2 \( \log \frac{a}{|x|} \)^B \,dx \)^{\frac{1}{2}} \( \log \frac{a}{|x|} \)^{\frac{1-B}{2}}\,\, (u \in C_{c, {\rm rad}}^\infty (B_1))
\end{align}
By using (\ref{Hardy with log}) and (\ref{radial lemma with log}), we have
\begin{align*}
\int_{B_1} \frac{|u|^p}{|x|^2 (\log \frac{a}{|x|})^A} \,dx 
&\le C \int_{B_1} \frac{|u|^p}{|x|^2 (\log \frac{a}{|x|})^{1+ \frac{p}{2}(1-B) }} \,dx\\
&\le C \int_{B_1} \frac{|u|^{p-2} }{\(\log \frac{a}{|x|} \)^{\frac{(p-2)(1-B)}{2}}  } \,\frac{|u|^2}{|x|^2 (\log \frac{a}{|x|})^{2-B} }\,dx\\
&\le C \( \int_{B_1} \( \log \frac{a }{|x|} \)^{B} |\nabla u|^2 \,dx \)^{\frac{p-2}{2}} \int_{B_1} \frac{|u|^2}{|x|^2 (\log \frac{a}{|x|})^{2-B} }\,dx\\
&\le C \( \int_{B_1} \( \log \frac{a }{|x|} \)^{B} |\nabla u|^2 \,dx \)^{\frac{p}{2}}
\end{align*}

Therefore, the inequality \eqref{ineq log} holds for any radial functions.

%
%
\noindent
{\bf (II)}\hspace{1em} Let $0< B <1$.
In this case, we apply the weighted nonlinear potential theory by D. R. Adams to prove Theorem \ref{Thm ineq log}.
In order to do so, we need to choose appropriate weights to which Theorem A is applicable:
Set 
\begin{align}
\label{weight g}
	g(x) = \begin{cases}
	|x|^{-2} \( \log \frac{a}{|x|} \)^{-A}  &\text{if} \,\, x \in B_1,\\
	0 &\text{if} \,\, x \in \re^2 \setminus B_1,
	\end{cases}
\end{align}
and
\begin{align}
\label{weight omega}
	\w(x)  = \begin{cases}
	\( \log \frac{a}{|x|} \)^{B}  &\text{if} \,\, x \in B_1,\\
	|x|^\gamma \( \log a \)^{B} &\text{if} \,\,x \in \re^2 \setminus B_1,
	\end{cases}
\end{align}
where $0< \gamma <2$.
Note that $g \in L^1_{loc}(\re^2)$ and $\w \in C(\re^2 \setminus \{ 0\})$. 
Furthermore, we obtain the followings.
%
%
\begin{lemma}\label{claim1}
$\w$ belongs to Muckenhoupt $A_2$-class.
\end{lemma}
\begin{lemma}\label{claim2}
$\sup_{x \in \re^2, r >0} \( \int_{B_r (x)} g(x) \,dx \) J[\w](x,r)^{\frac{p}{2}} < \infty$.
\end{lemma}

For the moment, we assume the validity of Lemma \ref{claim1}, \ref{claim2}.
Then from Theorem A, Lemma \ref{claim1} and Lemma \ref{claim2},  
we see that the inequality
\begin{align}\label{Adams cond}
	\| \,I_1 * f \,\|_{L^p (\re^2 ; \,g  \,dx)} \le  C \| f \|_{L^2 (\re^2; \, \w \,dx)}
\end{align}
holds for any $f \in L^2 (\re^2  ; \w \,dx)$ and $p > 2$ where $I_1(x) = \frac{1}{2\pi} |x|^{-1}$. 
From the Sobolev's integral representation:
\begin{align*}
	u(x)= \frac{1}{2\pi} \int_{\re^2} \frac{\nabla u(y) \cdot (x-y)}{|x-y|^2} \,dy \quad \text{for}\,\, x \in \re^2,
\end{align*}
we have 
\begin{align*}
	|u(x) | \le  [I_1 * |\nabla u| ] (x) \quad \text{for}\,\, x \in \re^2.
\end{align*}
Combining this with (\ref{Adams cond}) for $f = |\nabla u|$ implies 
\begin{align*}
	 \( \int_{\Omega} \frac{|u|^p}{|x|^2 (\log \frac{a}{|x|})^A} \,dx \)^{\frac{2}{p}}  
	&= \| u \|^2_{L^p (\Omega ; g (x) \,dx)} \\
	&\le  \| \,I_1 * |\nabla u| \,\|^2_{L^p (\re^2 ; g (x) \,dx)} \\
	&\le  C \| \nabla u \|^2_{L^2 (\re^2; \w (x) \,dx)}
	= C \int_\Omega \( \log \frac{a }{|x|} \)^{B} |\nabla u|^2 \,dx.
\end{align*}
This proves Theorem \ref{Thm ineq log}.
\qed

From now on, we will prove Lemma \ref{claim1} and Lemma \ref{claim2}.
We start by showing the useful computational lemma.

%
%
\begin{lemma}\label{claim3}
Let $-1 < \alpha < 1$, $a>1$, and $R > 0$ such that $\frac{a}{R} > e$. 
Then there exists a positive constant $C$ such that
\begin{align*}
	\int_{B_R(0)} \( \log \frac{a}{|y|} \)^\alpha \,dy \le C R^2 \( \log \frac{a}{R} \)^\alpha.
\end{align*}
\end{lemma}

\begin{proof}
By a change of variables, we have
\begin{align*}
	&\int_{B_R(0)} \( \log \frac{a}{|y|} \)^\alpha \,dy = 2\pi \int_0^R  \( \log \frac{a}{s} \)^\alpha s \,ds  = 2\pi a^2 \int_{\log \frac{a}{R}}^{\infty} t^{\alpha} e^{-2t} dt.
\end{align*}
First, we consider the case where $\alpha \in (-1, 0]$. 
Since the function $g(t) = t^{\alpha}$ is monotone decreasing in $\( \log \frac{a}{R}, \infty \)$, we have
\begin{align*}
	2\pi a^2 \int_{\log \frac{a}{R}}^{\infty} t^{\alpha} e^{-2t} dt 
	&\le 2\pi a^2 \( \log \frac{a}{R} \)^{\alpha} \int_{\log \frac{a}{R}}^{\infty} e^{-2t} dt \\
	&= 2\pi a^2 \( \log \frac{a}{R} \)^{\alpha} \left[ \frac{e^{-2t}}{-2} \right]_{\log \frac{a}{R}}^{\infty} = \pi R^2 \( \log \frac{a}{R} \)^{\alpha}.
\end{align*}
Next, we consider the case where $\alpha \in (0, 1)$. Set $\beta = 2\log \frac{a}{R} > 2$. 
Since the function $g(t) = t^{\beta} e^{-2t}$ is monotone decreasing in $\( \log \frac{a}{R}, \infty \)$, we have
\begin{align*}
	2\pi a^2 \int_{\log \frac{a}{R}}^{\infty} t^{\alpha-\beta} \cdot t^{\beta} e^{-2t} dt 
	&\le 2\pi a^2 \( \log \frac{a}{R} \)^{\beta} e^{-\beta} \int_{\log \frac{a}{R}}^{\infty} t^{\alpha-\beta} dt \\
	&= 2\pi a^2 \( \log \frac{a}{R} \)^{\beta} \( \frac{a}{R} \)^{-2} \left[ \frac{t^{\alpha-\beta+1}}{\alpha-\beta+1} \right]_{\log \frac{a}{R}}^{\infty}\\
	&= 2\pi R^2 \( \log \frac{a}{R} \)^{\beta} \frac{\( \log \frac{a}{R} \)^{\alpha-\beta+1}}{\beta-(\alpha +1)} \\
	&= \frac{2\pi}{2\log \frac{a}{R}-(\alpha+1)} R^2 \( \log \frac{a}{R} \)^{\alpha+1} 
	\le \frac{2\pi}{1-\alpha} R^2 \( \log \frac{a}{R} \)^{\alpha},
\end{align*}
where we have used $\alpha - \beta + 1 < 0$ and the last inequality comes from $f(x) = \frac{x}{2x - (\alpha+1)} < \frac{1}{1-\alpha}$ for any $\alpha < 1 < x$. 
\end{proof}

%
%
\vspace{1em}\noindent
{\it Proof of Lemma \ref{claim1}}:
To show that $\w$ in \eqref{weight omega} belongs to Muckenhoupt $A_2$-class, we show
\begin{align}\label{Muckenhoupt}
	\sup_{x \in \re^2, r>0} S(x,r) = \sup_{x \in \re^2, r>0} \frac{1}{\pi^2 r^4} \( \int_{B_r(x)} \w (y) \,dy \) \( \int_{B_r(x)} \w (y)^{-1} \,dy \) < \infty.
\end{align}
Note that we are in the case $0 < B < 1$.
According to the value of $|x|$, we divide the proof into three parts.

\noindent
{\bf (I)} The case where $x=0$.

First, we assume that $r \le 1$. From Lemma \ref{claim3}, we have
\begin{align*}
	&S(0,r) = \frac{1}{\pi^2 r^4} \( \int_{B_r(0)} \( \log \frac{a}{|y|} \)^B \,dy \) \( \int_{B_r(0)} \( \log \frac{a}{|y|} \)^{-B} \,dy \) \\
	&\le \frac{1}{\pi^2 r^4} \( C R^2 \( \log \frac{a}{R} \)^B \)\( C R^2 \( \log \frac{a}{R} \)^{-B} \) \le C.
\end{align*} 
Next, we assume that $r >1$. Then we have
\begin{align*}
	S(0,r) &= \frac{1}{\pi^2 r^4} \( \int_{B_1(0)} \( \log \frac{a}{|y|} \)^B \,dy + \int_{B_r(0) \setminus B_1(0)} \( \log a \)^B |y|^\gamma \,dy \) \\
	&\times \( \int_{B_1(0)} \( \log \frac{a}{|y|} \)^{-B} \,dy + \int_{B_r(0) \setminus B_1(0)} \( \log a \)^{-B} |y|^{-\gamma} \,dy \) \\
	&\le \frac{C}{r^4} \( 1+ r^{2+ \gamma} \) \( 1+ r^{2- \gamma} \)
	\le \frac{C}{r^{2-\gamma}} + C < \infty
\end{align*}
since $0 < \gamma < 2$.

\noindent
{\bf (II)} The case where $0 < |x| \le 1$. 

Let $M>0$ be a large constant such that $M > \frac{1}{a-1}$. 
We divide the prof into four parts.

\noindent
{\bf (II-a)} 
The case where $0 < r \le \min \left\{ \frac{|x|}{M}, \, 1-|x| \right\}$.

Then we see that $y \in B_r (x) \subset B_{|x| +r}(0) \setminus B_{|x| -r} (0) \subset B_1$. 
Therefore, we have
\begin{align*}
	S(x,r) 
	&\le \frac{1}{\pi^2 r^4} \( \int_{B_r(x)} \( \log \frac{a}{|x|-r} \)^B \,dy \) \( \int_{B_r(x)} \( \log \frac{a}{|x|+r} \)^{-B} \,dy \) \\
	&= \( \frac{\log \frac{a}{|x|} + \log \frac{|x|}{|x|-r} }{\log \frac{a}{|x|} + \log \frac{|x|}{|x|+r} } \)^B 
	\le \( \frac{\log \frac{a}{|x|} + \log \frac{M}{M-1} }{\log \frac{a}{|x|} + \log \frac{M}{M+1} } \)^B < \infty
\end{align*} 
for any $x \in B_1 (0) \setminus \{ 0\}$. 
Here we have used $\frac{|x|}{|x|-r} \le \frac{M}{M-1}$ and $\frac{|x|}{|x|+r} \ge \frac{M}{M+1}$ for $r \le \frac{|x|}{M}$.

\noindent
{\bf (II-b)} 
The case where $1-|x| \le  r \le \frac{|x|}{M}$.

In this case, since $0 < |x| \le 1$, we have $0 \le r \le \frac{1}{M}$ and $\frac{M}{M+1} \le |x| \le 1$.
Then we see that $y \in B_r (x) \subset B_{|x| +r}(0) \setminus B_{|x| -r}(0)$ and
$|x| + r \le \frac{M+1}{M}$, $|x| - r \ge \frac{M-1}{M+1}$.
Therefore, we have
\begin{align*}
	S(x,r) &\le \frac{1}{\pi^2 r^4} \( \int_{B_r (x) \cap B_1(0)} \( \log \frac{a}{|x|-r} \)^B \,dy + \int_{B_r(x) \cap B_1(0)^c} \( \log a \)^B (|x| +r)^\gamma \,dy \) \\
	&\times \( \int_{B_r (x) \cap B_1(0)} \( \log \frac{a}{|x| + r} \)^{-B} \,dy + \int_{B_r(x) \cap B_1(0)^c} \( \log a \)^{-B}  \,dy \) \\
	&\le C \(  \( \log \frac{M+1}{M-1} a \)^B   + (\log a)^B \( \frac{M+1}{M} \)^\gamma \) \(  \( \log \frac{M}{M+1} a \)^{-B}  + (\log a)^{-B} \) 
	< \infty.
\end{align*}

\noindent
{\bf (II-c)} The case where $\frac{|x|}{M} \le r \le 1-|x|$.

In this case, $B_r(x) \subset B_{|x|+r}(0) \subset B_1(0)$. From Lemma \ref{claim3}, we have
\begin{align*}
	&S(x,r) \le \frac{1}{\pi^2 r^4} \( \int_{B_{|x|+r}(0)} \( \log \frac{a}{|y|} \)^B \,dy \) \( \int_{B_{|x|+r}(0)} \( \log \frac{a}{|y|} \)^{-B} \,dy \) \\ 
	&\le \frac{C}{r^4} (|x| + r)^2 \( \log \frac{a}{|x|+r} \)^B \cdot (|x| + r)^2 \( \log \frac{a}{|x|+r} \)^{-B} \\
	&\le C \( \frac{|x| + r}{r} \)^4 \le C \( \frac{Mr + r}{r} \)^4 < \infty.
\end{align*}

\noindent
{\bf (II-d)} The case where $r \ge \max \left\{ \frac{|x|}{M}, \, 1-|x| \right\}$.

Then we see that $r \ge \frac{1}{M+1}$ is away from $0$ and $|x| + r \le (M+1) r$. 
Therefore, by Lemma \ref{claim3}, we have
\begin{align*}
	S(x,r) &\le \frac{1}{\pi^2 r^4} \( \int_{B_1(0)} \( \log \frac{a}{|y|} \)^B \,dy + \int_{B_{|x|+r}(0) \setminus B_1(0)} \( \log a \)^B |y|^\gamma \,dy \) \\
	&\times \( \int_{B_1(0)} \( \log \frac{a}{|y|} \)^{-B} \,dy + \int_{B_{|x|+r}(0) \setminus B_1(0)} \( \log a \)^{-B} |y|^{-\gamma} \,dy \) \\
	&\le \frac{C}{r^4} \( (\log a)^B + (\log a)^B \int_0^{(M+1)r} s^{\gamma + 1} ds \) \cdot \( (\log a)^{-B} + (\log a)^{-B} \int_0^{(M+1)r} s^{-\gamma + 1} ds \) \\
	&\le \frac{C}{r^4} \( 1 + r^{\gamma +2}  \) \( 1 + r^{2- \gamma}  \) < \infty
\end{align*}
since $0 < \gamma < 2$.

\noindent
{\bf (III)} The case where $|x| > 1$.

Again, we divide the proof into four cases.

\noindent
{\bf (III-a)} The case where $0 < r \le \min \left\{ \frac{|x|}{M}, \, |x|-1 \right\}$.

Then we see that $B_r (x) \cap B_1 = \emptyset$. 
Since $B_{|x|-r}(0) \subset B_r(x) \subset B_{|x|+r}(0)$, we have
\begin{align*}
	S(x,r) 
	&= \frac{1}{\pi^2 r^4} \( \int_{B_r(x)} \( \log a \)^B |y|^\gamma \,dy \) \( \int_{B_r(x)} \( \log a \)^{-B} |y|^{-\gamma} \,dy \) \\
	&\le \( \frac{|x| +r}{|x|-r} \)^\gamma  
	= \( 1+ \frac{2}{\frac{|x|}{r} -1} \)^\gamma
	\le \( 1+ \frac{2}{M-1} \)^\gamma < \infty.
\end{align*} 

\noindent
{\bf (III-b)} The case where $|x|-1 \le  r \le \frac{|x|}{M}$. 

The proof is the same as it in the case (II-b). We omit the proof. 


\noindent
{\bf (III-c)} The case where $\frac{|x|}{M} \le  r \le |x|-1$.

Then we see that $B_r(x) \cap B_1 = \emptyset$. Therefore, we have
\begin{align*}
S(x,r) \le \frac{1}{\pi^2 r^4} \( \int_{B_{|x|+r}(0)} |y|^\gamma \,dy \) \( \int_{B_{|x|+r}(0)} |y|^{-\gamma} \,dy \) 
\le C \( \frac{r+|x|}{r} \)^4 = C \( 1+ \frac{1}{M} \) < \infty.
\end{align*}

\noindent
{\bf (III-d)} The case where $r \ge \max \left\{ \frac{|x|}{M}, \, |x|-1 \right\}$.

The proof is the same as it in the case (II-d). We omit the proof. 


Finally, we obtain (\ref{Muckenhoupt}), which ends the proof of Lemma \ref{claim1}. 
\qed

%
%


\vspace{1em}\noindent
{\it Proof of Lemma \ref{claim2}}:
Set 
\begin{align*}
	T(x, r) &= \int_{B_r (x) \cap B_1 (0)} g (y) \,dy = \int_{B_r (x) \cap B_1 (0)} \frac{dy}{|y|^2 \( \log \frac{a}{|y|} \)^A},\\
	J(x, r) &= J[\w](x,r) = \int_r^\infty  \frac{1}{\pi t^2} \( \int_{B_t (x)} \frac{dy}{\w (y)} \) \, \frac{dt}{t},
\end{align*}
where $g$ and $\w$ are defined in \eqref{weight g} and \eqref{weight omega}.
Our goal is to show 
\begin{align}\label{TJ}
	\sup_{x \in \re^2, \,r >0} T(x,r) J(x, r)^{\frac{p}{2}} < \infty.
\end{align}
According to the value of $r >0$, we divide the proof into two parts.

%

\noindent
{\bf (I)} The case where $0 < r \le \frac{1}{2}$.

\noindent
{\bf (I-a)} The case where $r \le \frac{|x|}{2}$ and $|x| < \frac{2}{3}$. 

Then we see that $\frac{|x|}{2} < |x| - r < |y| < |x| + r < \frac{3}{2} |x| < 1$ for $y \in B_r (x)$. 
We fix $\tilde{a} > 1$ large enough so that the function $|y|^{-2} \( \log \frac{\tilde{a}}{|y|} \)^{-A}$ is radially decreasing on $B_1$. 
Then using $\frac{|x|}{2} < |y|$, we have
\begin{align*}
	T(x, r) \le C \int_{B_r (x)} \frac{dy}{|y|^2 \( \log \frac{\tilde{a}}{|y|} \)^A} \le \frac{C r^2}{|x|^2 \( \log \frac{2 \tilde{a}}{|x|} \)^A}.
\end{align*}
Set 
\begin{align*}
	J(x, r) &= \int_r^{\frac{|x|}{2}}  \frac{1}{\pi t^2} \( \int_{B_t (x)} \frac{dy}{\w (y)} \) \, \frac{dt}{t} + \int_{\frac{|x|}{2}}^\infty  \frac{1}{\pi t^2} \( \int_{B_t (x)} \frac{dy}{\w (y)} \) \, \frac{dt}{t} \\
	&=: J_1 (x, r) + J_2 (x).
\end{align*}
For $J_1$, since $\( \log \frac{a}{|y|} \)^{-B} \le \( \log \frac{2a}{3|x|} \)^{-B}$ for $|y| < \frac{3}{2} |x|$, we have
\begin{align}
\label{J_1 estimate}
	J_1(x, r) &= \int_r^{\frac{|x|}{2}}  \frac{1}{\pi t^2} \( \int_{B_t (x)} \frac{dy}{\( \log \frac{a}{|y|} \)^B} \) \, \frac{dt}{t}
	\le \( \log \frac{2a}{3|x|} \)^{-B} \int_r^{\frac{|x|}{2}} \frac{dt}{t} \notag \\ 
	&= \( \log \frac{2a}{3|x|} \)^{-B} \log \frac{|x|}{2r}.
\end{align}
For $J_2$, note that $B_t(x) \subset B_1(0)$ if $0 < t < \frac{1}{3}$, since $|x| < \frac{2}{3}$.
Also note that 
\begin{align*}
	&\int_{B_1 (0)} \frac{dy}{\( \log \frac{a}{|y|} \)^B} \le 2\pi (\log a)^{-B} < \infty, \quad \text{for }  a > 1, B > 0, \\
	&\int_{B_R(0) \cap B_1(0)^c} \frac{dy}{|y|^\gamma \( \log a \)^B} \le (\log a)^{-B} \int_1^R s^{1-\gamma} \, ds \quad \text{for any } R > 1. 
\end{align*}
Thus we have
\begin{align}
\label{J_2 estimate}
	J_2 (x) &= \int_{\frac{|x|}{2}}^\infty  \frac{1}{\pi t^2} \( \int_{B_t (x) \cap B_1 (0)} \frac{dy}{\( \log \frac{a}{|y|} \)^B} 
	+ \int_{B_t (x) \cap B_1 (0)^c} \frac{dy}{|y|^\gamma \( \log a \)^B} \) \, \frac{dt}{t} \notag \\
	&\le \int_{\frac{|x|}{2}}^{\frac{1}{3}} \frac{1}{\pi t^2} \( \int_{B_t (x)} \frac{dy}{\( \log \frac{a}{|y|} \)^B} \) \, \frac{dt}{t} \notag \\
	&+ \int_{\frac{1}{3}}^{\infty} \frac{1}{\pi t^2} \( \int_{B_t (x) \cap B_1 (0)} \frac{dy}{\( \log \frac{a}{|y|} \)^B} + \int_{B_{t+|x|}(0) \cap B_1 (0)^c} \frac{dy}{|y|^\gamma \( \log a \)^B} \) \, \frac{dt}{t} \notag  \\
	&\le \( \log \frac{2a}{3|x|} \)^{-B} \int_{\frac{|x|}{2}}^{\frac{1}{3}} \, \frac{dt}{t} + C (\log a)^{-B} \int_{\frac{1}{3}}^\infty \( 1 +  \int_1^{|x| + t} s^{1-\gamma} \,ds \) \frac{dt}{t^3} \notag \\
	&\le \( \log \frac{2a}{3|x|} \)^{1-B} + C
\end{align}
since $0 < B < 1$ and $0 < \gamma < 2$.
Combining \eqref{J_1 estimate} and \eqref{J_2 estimate}, we obtain
\begin{align*}
	T(x,r) J(x, r)^{\frac{p}{2}} 
	&\le 2^{\frac{p-2}{2}} T(x,r) \(  J_1(x, r)^{\frac{p}{2}} +  J_2(x, r)^{\frac{p}{2}} \) \\
	&\le \frac{C r^2 \( \log \frac{|x|}{2r} \)^{\frac{p}{2}}}{|x|^2 \( \log \frac{2\tilde{a}}{|x|} \)^A  \( \log \frac{2a}{3|x|} \)^{\frac{Bp}{2}}} 
	+\frac{C r^2 \( \log \frac{2a}{3|x|} \)^{\frac{p}{2} (1-B) }}{|x|^2 \( \log \frac{2\tilde{a}}{|x|} \)^A }\\
	&\le C \( \frac{r}{|x|} \)^2 \( \log \frac{|x|}{2r} \)^{\frac{p}{2}} + C,
\end{align*}
where the last inequality follows from the assumption $A \ge 1 + \frac{p}{2} (1-B) \ge \frac{p}{2} (1-B)$ in Theorem \ref{Thm ineq log}. 
In the case (I-a), we have $\frac{r}{|x|} < 1/2$.
Since the function $t \mapsto t^2 (\log \frac{1}{2t})^{p/2}$ is bounded for $0 < t < 1/2$, we have the result \eqref{TJ} in this case.


\noindent
{\bf (I-b)} The case where $r \le \frac{|x|}{2}$ and $\frac{2}{3} < |x| < 1$. 

If $\frac{1}{3} < r$, then we easily show that $T(x,r) J(x,r)^{\frac{p}{2}} \le C$ for some $C > 0$ independent of $x \in \re^2$ and $r > \frac{1}{3}$. 
Therefore, we assume that $r \le \frac{1}{3}$.
Since $\frac{1}{3} < |x| -r \le |x| - t < |x| + t \le |x| + r < \frac{4}{3}$ for $t \in [r, \frac{1}{3}]$, 
we have, as in the former case,
\begin{align*}
	T(x, r) &\le C r^2,\\
	J(x, r) &\le \int_r^{\frac{1}{3}} 2 \( \log a \)^{-B} \, \frac{dt}{t} + C (\log a)^{-B} \int_{\frac{1}{3}}^\infty  \frac{1}{\pi t^2} \( 1 + \int_1^{1+t} s^{1-\gamma} \,ds \) \, \frac{dt}{t} \\
	&\le C \( \log \frac{1}{3r} + 1 \).
\end{align*}
Thus we obtain $T(x,r) J(x,r)^{\frac{p}{2}} \le C r^2 \( \log \frac{1}{3r} + 1  \)^{\frac{p}{2}} < \infty$ for any $x \in \re^2$ and $0 < r \le \frac{1}{3}$.

%

\noindent
{\bf (I-c)} The case where $r \le \frac{|x|}{2}$ and $|x| \ge 1$.

If $r \le |x| -1$, then $T(x, r) =0$. 
Therefore, we assume that $r > |x| -1$. Then we see that $|x| + t < r + 1 + t \le \frac{3}{2} + t$. 
Therefore, we have
\begin{align*}
	T(x, r) &\le C r^2,\\
	J(x, r) &\le \int_r^{\frac{1}{2}} 2 \( \log a \)^{-B} \, \frac{dt}{t} + C (\log a)^{-B} \int_{\frac{1}{2}}^\infty  \frac{1}{\pi t^2} \( 1 + \int_1^{\frac{3}{2} + t} s^{1-\gamma} \,ds \) \, \frac{dt}{t} \\
	&\le C \( \log \frac{1}{2r} + 1 \).
\end{align*}
Thus we obtain $T(x,r) J(x,r)^{\frac{p}{2}} \le C r^2 \( \log \frac{1}{2r} + 1  \)^{\frac{p}{2}} < \infty$ for any $x \in \re^2$ and $0 < r \le \frac{1}{2}$.

%
\noindent
{\bf (I-d)} The case where $\frac{|x|}{2} < r$.

We may assume that $r < \frac{1}{3}$. Then we see that $B_r (x) \subset B_{3r} (0) \subset B_1 (0)$. First, we consider the case where $|x| \le \frac{1}{2}$. From Lemma \ref{claim3}, we have
\begin{align*}
	T(x, r) &\le \int_{B_{3r} (0)} \frac{dy}{|y|^2 \( \log \frac{a}{|y|} \)^A} = \frac{2\pi}{A-1} \( \log \frac{a}{3r} \)^{1-A},\\
	J(x, r) &\le \int_r^{1-|x|} \frac{1}{\pi t^2} \int_{B_{t+|x|} (0)} \( \log \frac{a}{|y|} \)^{-B} \, \frac{dt}{t} + C \int_{1-|x|}^\infty  \frac{1}{\pi t^2} \( 1 + \int_1^{|x|+t} s^{1-\gamma} \,ds \) \, \frac{dt}{t}\\
	&\le C \int_r^{1-|x|} \( \frac{t+|x|}{t} \)^2  \( \log \frac{a}{|x|+t} \)^{-B} \, \frac{dt}{t} + C \int_{\frac{1}{2}}^\infty  \frac{1}{\pi t^2} \( 1 +  (t+1)^{2-\gamma} \) \, \frac{dt}{t} \\
	&\le C \int_r^{1-|x|} 3^3 \( \log \frac{a}{|x|+t} \)^{-B} \, \frac{dt}{t+ |x|} + C 
	\le C \( \( \log \frac{a}{r} \)^{1-B} + 1 \),
\end{align*}
Thus we have
\begin{equation*}
	T(x,r) J(x,r)^{\frac{p}{2}} \le C \( \log \frac{a}{3r} \)^{1-A} \( \( \log \frac{a}{r} \)^{\frac{p}{2} (1-B)} + 1 \).
\end{equation*}
The last expression is finite since $A \ge 1 + \frac{p}{2} (1-B)$. 
We can also show the case where $\frac{1}{2} < |x| < 1$ in the same way as above, so we omit the proof.  

\noindent
{\bf (II)} The case where $r \ge \frac{1}{2}$.

If $|x|< 2$, then we can easily see that $T(x, r)$ and $J(x,r)$ are finite. Also, if $r < |x| -1$, then $T(x, r) =0$. Therefore, we assume that $|x| \ge 2$ and $r \ge |x|-1$. 
In this case, we can also show in the same way as (I-c) that $T(x, r)$ and $J(x,r)$ are finite. 
Therefore, we omit the proof. 

Finally, we obtain (\ref{TJ}), which ends the proof of Lemma \ref{claim2}. 
\qed

\section*{Note added in the proof.}
After the completion of the manuscript, we have been informed by Prof. T. Horiuchi (Ibaraki University)
that the inequality in Theorem \ref{Thm ineq log} is a special case of a series of weighted inequalities proved in his recent paper \cite{H} (Theorem 3.1.). 
His method is different from that in our paper and do not exploit the weighted nonlinear potential theory.
We thank Prof. Horiuchi for informing us of the fact.


%
%


\section*{Acknowledgment}
The first author (M.S.) was supported by JSPS KAKENHI Early-Career Scientists, No. JP19K14568. 
The second author (F.T.) was supported by JSPS Grant-in-Aid for Scientific Research (B), No. JP19136384. 
This work was partly supported by Osaka City University Advanced Mathematical Institute (MEXT Joint Usage/Research Center on Mathematics and Theoretical Physics).


\end{document}